\documentclass[11pt,a4paper]{article}
\usepackage{amsmath}
\usepackage{amssymb}
\usepackage{amsfonts}
\usepackage{amsthm}
\usepackage{graphics}
\usepackage{graphicx}%
\usepackage{wrapfig}
\numberwithin{equation}{section}
\newtheorem{remark}{Remark}[section]

\newtheorem{theorem}{Theorem}[section]
\newtheorem{corollary}{Corollary}[section]
\newtheorem{definition}{Definition}[section]
\newtheorem{prop}{Proposition}[section]

\begin{document}

\begin{center}
{\LARGE \bf
Asymptotic Approximation for the Solution
\\[2mm]
 to a Semi-linear Parabolic Problem
\\[5mm]
in a Thick Fractal Junction
}
\end{center}
\vskip50pt
\begin{center}
{\Large \bf  Taras A. Mel'nyk}
\end{center}
\vskip20pt
\begin{center}\vskip-25pt
{\footnotesize
Department of Mathematical Physics,
\\
 Faculty of Mechanics and Mathematics,
 \\
 Taras Shevchenko National University of Kyiv ,\\
Volodymyrska st., 64/13\\
Kyiv 01601, Ukraine,\\
{\tt melnyk@imath.kiev.ua}}
\end{center}
\vskip20pt

\begin{abstract}
We consider a semi-linear parabolic problem in a model plane thick fractal junction $\Omega_{\varepsilon}$, which
is the union of a domain $\Omega_{0}$ and a lot of joined thin trees situated $\varepsilon$-periodically  along  some interval  on the boundary of
$\Omega_{0}.$ The trees have finite number of branching levels. The following nonlinear Robin boundary condition
$\partial_{\nu}v_{\varepsilon} +  \varepsilon^{\alpha_i} \kappa_i(v_{\varepsilon}) = \varepsilon^{\beta_i} g^{(i)}_{\varepsilon}$
is given on the boundaries of the branches from the $i$-th branching layer; $\alpha_i$ and $\beta_i$ are real parameters.
The asymptotic analysis of this problem is made as $\varepsilon\to0,$ i.e., when the number of the thin trees infinitely increases and their thickness vanishes. In particular, the corresponding homogenized problem is found  and
the existence and uniqueness of its solution in an anizotropic Sobolev space of multi-sheeted functions is proved.
We construct the asymptotic approximation for the solution $v_\varepsilon$  and prove
the corresponding asymptotic estimate in the space
$C\big([0,T]; L^2(\Omega_\varepsilon) \big) \cap L^2\big(0, T; H^1(\Omega_\varepsilon)\big)$, which shows the influence of  the parameters $\{\alpha_i\}$ and $\{\beta_i\}$  on the asymptotic behavior of the solution.
\end{abstract}

\bigskip

{\bf Comments:}\  29 pages, 4 figures

\bigskip

{\bf Subj-class:} \ Analysis of PDEs

\bigskip

{\bf MSC-classification:} \ 35B27, 74K30, 35B40, 35K57, 49J27

\bigskip


\newpage

\tableofcontents
\newpage

\pagestyle{myheadings}
\markboth{}{{\underline{T.A. Mel'nyk, \hfill{}  Asymptotic approximation for the solution}}}

\section{Introduction}\label{Sec1}

In recent years, materials with complex structure are widely used in  engineering devices in many fields of science.
It is known that some properties of materials are controlled by their geometrical structure. Therefore, the study of the influence of the material microstructure can improve its useful properties  and reduce undesirable effects.
The main methods for this study are asymptotic methods for boundary value problems (BVP's) in domains with complex structure: perforated domains, grid-domains, domains with rapidly oscillating boundaries, thick junctions, etc.

In this paper, we begin to study asymptotic properties of solutions to BVP's in thick junctions of a new type, namely {\it thick fractal junctions}.
A  thick fractal junction is  the  union of some domain, which is called the junction's body,
and a lot of joined thin trees situated $\varepsilon$-periodically  along  some manifold on the boundary of the junction's body.
The trees have finite number of branching levels. The small parameter $\varepsilon $ characterizes the distance between neighboring thin branches and also their thickness. On Fig.~\ref{f1} you can see a heat radiator with a fractal-structure that has one branching level.

\begin{figure}[htbp]
\begin{center}
\includegraphics[width=6cm]{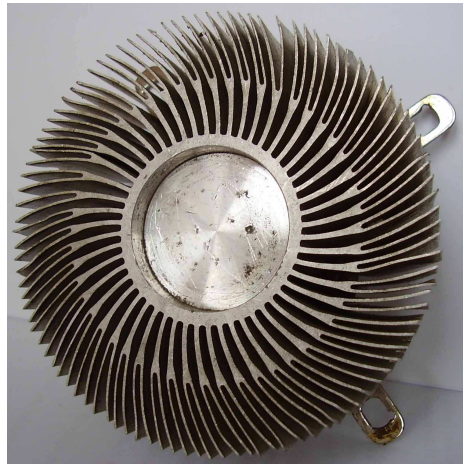}
\end{center}
\caption{Heat radiator shaped like a thick fractal junction}\label{f1}
\end{figure}

Various constructions of thick junction type are successfully used in nanotechnologies \cite{Lenczner}, microtechnique \cite{Lyshevshi}, modern engineering constructions (microstrip radiator, ferrite-filled rod radiator), as well as many physical and biological systems.
 For example, a number of new applications are envisioned, especially regarding efficient sensors (inertial, biological, chemical), signal processing filters (ultra large band), micro-fractal constructions:  fractal antennas, fractal transistors, fractal heat radiators and so on.

Such successful applications of thick-junction constructions have stimulated active learning BVP's in thick junctions with more complex structures:  thick junctions with the thin junction's body \cite{B-G_2003,B-G-Mos_2007,B-G-M},
thick multi-level junctions \cite{M-U-T,Mel-Durante-2012,Mel-Sad-2013}, thick cascade junctions \cite{Mel-Chech-2014,Mel-Che-2008}, where new qualitative results were obtained. Specifically, it was shown that processes in thick multi-level junctions behave as a “many-phase system” and
thick cascade junctions have new kind of eigenvibrations. This means that materials with such micro-structures have some new properties.

Designing such arrays of mechanical components in thick junctions cannot be achieved with today softwares, because this would require too much CPU resources. Regarding their number of components (in some cases few thousands), development of new mathematical tools are necessary. One of them is asymptotic analysis of BVP's in thick junctions as $\varepsilon \to 0 ,$ i.e., when the number of attached thin domains infinitely increases and their thickness decreases to zero. Asymptotic results give us the possibility to replace the original problem in a thick junction by the corresponding homogenized problem that is more simpler and then apply computer simulation. In addition, in some cases it is possible to construct accurate and numerically implementable asymptotic approximations.

As a first step, here we consider a nonlinear boundary-value problem  for a reaction-diffusion equation
in a model $2D$ thick fractal junction $\Omega_\varepsilon$ (see Fig~\ref{f2}).
Of course, it is possible to consider a thick fractal junction that has more complex branching structures.
However, the main features in the asymptotic behavior of solutions to BVP's in thick fractal junctions can be observed
on the example of  $\Omega_\varepsilon$ (a thick fractal junction with two branching levels).

The rest of this paper is organized as follows. 

 The statement of the problem and features of the investigation are given in  Section~\ref{sec1}.

 In Section~\ref{Sec2} we formally construct the leading terms of asymptotic expansions for a solution to our problem.
The asymptotics consists of the outer expansions both in the junction's body and in each thin branches as well as the leading terms of inner expansions in a neighborhood both of the joint zone and each branching levels.

Then in Section~\ref{Sec3}, using the method of matched asymptotic expansions, we derive the corresponding nonstandard homogenized problem. The existence and uniqueness of its solution in an anizotropic Sobolev space of multi-sheeted functions is proved in Section~\ref{hom-problem}.

In Section~\ref{approx} we construct an approximating function, find its residuals, estimate them and prove
the main asymptotic estimate for the difference between the solution and the approximating function.

\section{Statement of the problem}\label{sec1}

Let $\Omega_0$ be a bounded domain  in  $\Bbb R^2$ with the Lipschitz boundary $\partial\Omega_0$ and $\Omega_0\subset \{
x:=(x_1, x_2)  \in \Bbb R^2 : \ x_2 > 0\}.$ Let $\partial\Omega_0$ contain the segment
$I_0=\{x: x_1\in [0,a], \ \ x_2=0\}.$ We also assume that there exists a positive
 number $\delta_0$ such that $\Omega_0 \cap \{x: \  0 < x_2 < \delta_0\}= \{x: \ x_1\in (0, a), \ x_2 \in (0, \delta_0)\}.$

Let $a, l_1, l_2, l_3$ be positive numbers, $h_0, h_{1,1}, h_{1,2}, h_{2,1}, h_{2,2}, h_{2,3}, h_{2,4} $ be  fixed numbers from the interval $(0,1)$ and $ h_{1,1} + h_{1,2} < h_0,$  $h_{2,1} + h_{2,2} < h_{1,1},$ $h_{2,3} + h_{2,4} < h_{1,2}.$ Let us also introduce  a small parameter $\varepsilon= \frac{a}{N},$ where $N$ is  a large positive integer.

\begin{figure}[htbp] 
\begin{center}
\includegraphics[width=8cm]{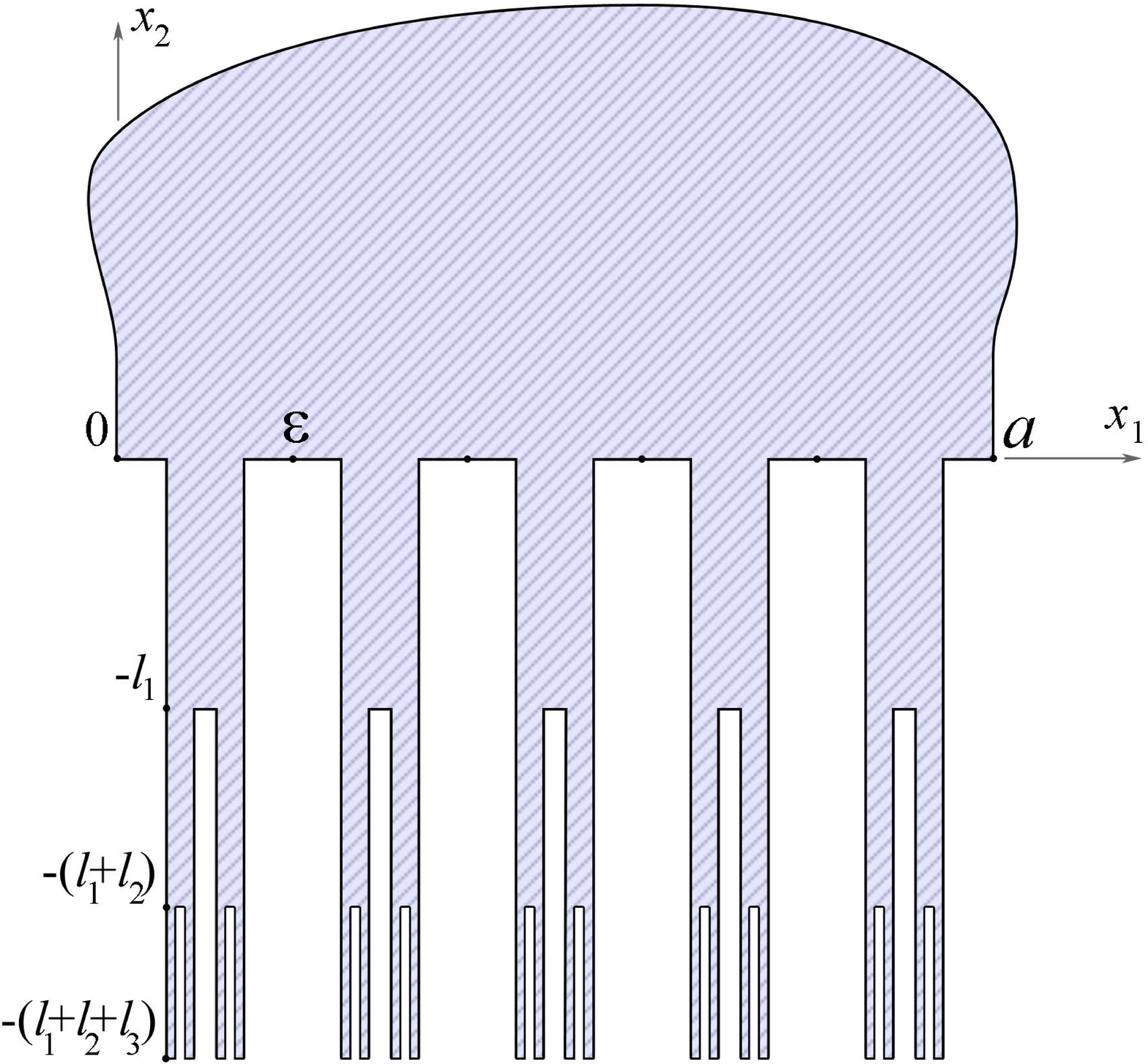}
\end{center}
\caption{A model  thick fractal junction $\Omega_{\varepsilon}$}\label{f2}
\end{figure}

A model thick fractal junction $\Omega_{\varepsilon}$ (see Fig.~\ref{f2}) consists of the junction's body $\Omega_{0},$
\begin{itemize}
\item
a large number of the thin rods $G^{(0)}_\varepsilon = \bigcup_{j=1}^{N-1} G^{(0)}_{j}(\varepsilon),$
$$
G^{(0)}_{j}(\varepsilon)= \left\{ x: \, \left| x_{1} - \varepsilon (j+\frac{1}{2})\right|<\frac{\varepsilon h_0}{2}, \quad x_{2}  \in (-l_1, 0]
\right\},
$$
from the zero layer,
  \item
  a large number of the thin rods $G^{(1,m)}_\varepsilon = \bigcup_{j=1}^{N-1} G^{(1,m)}_{j}(\varepsilon),$
$$
G^{(1,m)}_{j}(\varepsilon)= \left\{ x: \, \left| x_{1} - \varepsilon (j+ b_{1,m})\right|<\frac{\varepsilon h_{1,m}}{2}, \quad x_{2}  \in \big(- l_2 - l_1 , -l_1\big]
\right\},
$$
 from the first branching layer, where $m\in \{1, 2\}$ and
\begin{equation}\label{b1}
    b_{1,1} = \frac{1 - h_0 +h_{1,1}}{2} ,  \quad b_{1,2} = \frac{1 + h_0 - h_{1,2}}{2},
\end{equation}
  \item
and a large number of the thin rods $G^{(2,m)}_\varepsilon = \bigcup_{j=1}^{N-1} G^{(2,m)}_{j}(\varepsilon),$
$$
G^{(2,m)}_{j}(\varepsilon)= \left\{ x: \, \left| x_{1} - \varepsilon (j+ b_{2,m})\right|<\frac{\varepsilon h_{2,m}}{2},
\ \  x_{2}  \in \big(- l_3 - l_2 - l_1, - l_2 - l_1\big] \right\},
$$
\end{itemize}
 from the second branching layer, where $m\in \{1, 2, 3, 4\}$ and
\begin{gather}\label{b21}
b_{2,1} = \frac{1 - h_0 + h_{2,1}}{2} , \quad b_{2,2} = \frac{1 - h_0 + 2 h_{1,1} - h_{2,2}}{2},
\\
b_{2,3} = \frac{1 + h_0 - 2 h_{1,2} + h_{2,3}}{2}, \quad b_{2,4} = \frac{1 + h_0 - h_{2,4}}{2}. \label{b22}
\end{gather}

Thus,
$
\Omega_{\varepsilon}=\Omega_{0}\bigcup G^{(0)}_{\varepsilon} \bigcup G^{(1)}_{\varepsilon}
\bigcup G^{(2)}_{\varepsilon},
$
where
$
G^{(1)}_{\varepsilon} =  \bigcup_{m=1}^2 G^{(1,m)}_{\varepsilon},
$
$
G^{(2)}_{\varepsilon}= \bigcup_{m=1}^4 G^{(2,m)}_{\varepsilon}.
$
The small parameter $\varepsilon $ characterizes the distance between neighboring thin branches and also
their thickness. Precisely, each branch $G^{(i,m)}_{j}(\varepsilon)$  has small cross-section of size ${\cal O}(\varepsilon)$
and constant height. In addition,  at fixed $j\in \{0,1,\ldots,N-1\}$ branches $G^{(0)}_{j}(\varepsilon),$ $\{G^{(1,m)}_{j}(\varepsilon)\}_{m=1}^2,$
$\{G^{(2,m)}_{j}(\varepsilon)\}_{m=1}^4$  form the tree  with two branching levels.
These trees are $\varepsilon$-periodically distributed along the segment~$I_0.$

In $\Omega_\varepsilon$ we consider  the following semilinear parabolic
 initial boundary-value problem:
\begin{equation}\label{start_prob}
\left\{
    \begin{array}{rcll}
    \partial_t v_{\varepsilon} - \Delta v_{\varepsilon} + k(v_{\varepsilon}) &=& f_0
                    &\mbox{in} \ \Omega_0 \times (0,T) ,
                    \\[2mm]
\partial_t v_{\varepsilon} - \Delta v_{\varepsilon} + k_i(v_{\varepsilon}) &=& 0
                    &\mbox{in} \ G^{(i)}_{\varepsilon} \times (0,T) , \quad i=0, 1, 2,
                    \\[2mm]
 \partial_{\nu} v_{\varepsilon} + \varepsilon^{\alpha_i} \kappa_i(v_{\varepsilon}) &=& \varepsilon^{\beta_i} g^{(i)}_{\varepsilon}
                    &\mbox{on} \ \Upsilon^{(i)}_{\varepsilon} \times (0,T), \quad i=0, 1, 2,
                    \\[2mm]
    \partial_{\nu} v_{\varepsilon} &=& 0
                    &\mbox{on} \ \partial\Omega_{\varepsilon}\setminus \Big( \bigcup_{i=0}^2 \Upsilon^{(i)}_{\varepsilon}\Big)  \times (0,T) ,
 \\[2mm]
\left[v_{\varepsilon} \right]_{\big|{x_2= - \sum_{n=0}^i l_n}} &= & \left[\partial_{x_2} v_{\varepsilon} \right]_{\big|{x_2= - \sum_{n=0}^i l_n}}
=0  &
\mbox{on} \ Q^{(i)}_{\varepsilon}\times (0,T), \ \ i=0, 1, 2,
    \\[4mm]
    \left. v_{\varepsilon} \right|_{t=0} &=&0 &\mbox{in} \   \Omega_{\varepsilon},
   \end{array}\right.
\end{equation}
where $\partial_{\nu}$ is the outward normal derivative;  for each index $i\in \{0, 1, 2\}$
parameters $\alpha_i$ and $\beta_i$ are greater or equal $1,$
$\Upsilon^{(i)}_{\varepsilon}= \bigcup_{m=1}^{2i} \Upsilon^{(i,m)}_{\varepsilon},$ $\Upsilon^{(i,m)}_{\varepsilon} $ is the union of vertical boundaries of the thin rods  $G^{(i,m)}_{\varepsilon},$
$Q^{(i)}_{\varepsilon} = G^{(i)}_{\varepsilon} \cap \{x_2= - \sum_{n=0}^i l_n\},$ $l_0=0,$
$f_0,$ $g^{(i)}_{\varepsilon},$ $k,$ $k_i,$ $\kappa_i$  are given functions;
\ the brackets denote the jump of the enclosed quantities.

\begin{remark}\label{remark_1}
Hereafter we use the following shortening: $\{x_2= - \sum_{n=0}^i l_n\} := \{x\in \Bbb R^2: \ x_2= - \sum_{n=0}^i l_n\};$
also if the index $i=0,$ then the index $m$ is absent and notation as $\Upsilon^{(0,m)}_{\varepsilon} $ means $\Upsilon^{(0)}_{\varepsilon}.$
\end{remark}

Assumptions for the given functions are as follows.
The function $f_0$ belongs to the space $L^2(\Omega_0 \times(0,T))$ and its support is compactly embedded in $\Omega_0$
for a.e.  $t\in (0, T).$
The functions  $\{g^{(i)}_\varepsilon\}_{i=0}^2$ satisfy the following conditions:
\begin{itemize}
\item
$g^{(i)}_\varepsilon\in L^2(D_i\times(0,T)),$
where the domain
\begin{equation}\label{1.1}
D_i=\Big\{x : \ 0<x_1< a, \ \    - \sum\nolimits_{n=0}^{i+1} l_n < x_2 < - \sum\nolimits_{n=0}^i l_n \Big\}
\end{equation}
is filled up by the thin rods from the $i$-th layer in the limit
passage as $\varepsilon\to 0;$
\item
 there exist weak derivatives $\partial_{x_1}g^{(i)}_{\varepsilon}\in L^2(D_i \times(0,T)), \ i=0, 1, 2,$ and
 constants $ c_i, \,\varepsilon_0$ such that for each value $\varepsilon \in (0,\varepsilon_0)$
\begin{equation}\label{1.2}
\|g^{(i)}_{\varepsilon}\|_{L^2(D_i \times(0,T))} \,+\,
\|\partial_{x_1}g^{(i)}_{\varepsilon}\|_{L^2(D_i \times(0,T))} \,\le\, c_i ;
\end{equation}
\item  moreover, if $\ \beta_i=1,$ then there exists a function $g^{(i)}_0 \in L^2(D_i\times(0,T))$ such that
\begin{equation}\label{1.3}
g^{(i)}_{\varepsilon} \to g^{(i)}_0 \ \ \mbox{in} \ \ L^2(D_i\times(0,T))
\quad \mbox{as} \ \ \varepsilon \to 0.
\end{equation}
\end{itemize}

The functions  $k:\mathbb{R} \to \mathbb{R},$ $k_i:\mathbb{R} \to \mathbb{R},$ and $\kappa_i:\mathbb{R} \to \mathbb{R},$
 $i=0, 1, 2$ are continuously differentiable and
\begin{gather} \label{cond-2}
\exists \,c_1, \, c_2>0 : \ \
c_1\le k' \le c_2, \quad c_1\le k'_i \le c_2,
\quad c_1\le \kappa'_i \le c_2 \ \ \
\text{ in} \ \mathbb{R},
\\
 i=0, 1, 2. \notag
\end{gather}
From (\ref{cond-2}) it follows (see e.g. \cite{Mel-m2as-08})  the following inequalities:
\begin{gather}\label{est-1}
c_1 s^2 + k(0) s \le k(s) \, s \le c_2 s^2 + k(0)s,
\\[2mm]
\label{est-2'}
\exists \,c_3>0  \ \ \forall \, p, \, s  \in \Bbb R : \quad |k(p) - k(s)| \le  c_3 |p - s|,
\quad  |k(s)| \le  c_3 (1 +  |s|)
\end{gather}
(the same inequalities for the other functions $\{k_i\}, \{\kappa_i\}).$

Recall that a function $v_{\varepsilon }\in L^{2}\bigl(0,T;\, H^{1}(\Omega _{\varepsilon })\bigr),$
with $v'_\varepsilon  \in L^2\bigl(0,T;( H^{1}(\Omega_\varepsilon))^*\bigr),$
is a weak solution to problem~(\ref{start_prob}) if
\begin{equation}\label{1.4}
\langle v'_\varepsilon, \psi   \rangle_\varepsilon
+ \langle {\cal A}_\varepsilon(t) v_{\varepsilon} , \psi \rangle_\varepsilon  = \langle F_\varepsilon(t) , \psi \rangle_\varepsilon
\end{equation}
for each $\psi \in H^1(\Omega_\varepsilon)$ and a.e. $t\in (0, T),$ and $v_\varepsilon|_{t=0}=0.$

Here $\partial_t v_\varepsilon:= v'_\varepsilon,$ the brackets $\langle \cdot , \cdot \rangle_\varepsilon$ denotes the pairing of $H^1(\Omega_\varepsilon)^*$ with $H^1(\Omega_\varepsilon),$
the operator ${\cal A}_\varepsilon(t):  H^1(\Omega_\varepsilon) \mapsto H^1(\Omega_\varepsilon)^*$ is defined by the formula
$$
\langle {\cal A}_\varepsilon(t) v , \psi  \rangle_\varepsilon :=
\int\limits_{\Omega _{\varepsilon }}\nabla_{x} v \cdot \nabla _{x} \psi \,dx + \int\limits_{\Omega_0} k(v) \, \psi \, dx +
\sum_{i=0}^2 \int\limits_{G^{(i)}_{\varepsilon}} k_i(v) \, \psi \, dx +
  \varepsilon^{\alpha_i} \int\limits_{\Upsilon_{\varepsilon }^{(i)}} \kappa_i(v) \, \psi \, dx_2
$$
for all $ v, \psi \in H^1(\Omega_\varepsilon),$ and the linear functional
$F_\varepsilon(t) \in H^1(\Omega_\varepsilon)^*$ is defined as follows:
$$
\langle F_\varepsilon(t) , \psi \rangle_\varepsilon :=
\int\limits_{\Omega _{0}}f_{0}\,\psi \, dx + \sum_{i=0}^2 \varepsilon ^{\beta_i}\int\limits_{\Upsilon_{\varepsilon }^{(i)}}g^{(i)}_{\varepsilon }\, \psi \, dx_2, \qquad \forall \psi \in H^1(\Omega_\varepsilon),
$$
for a.e. $t\in [0,T].$ In addition, it is known that $v_{\varepsilon } \in C([0,T]; L^2(\Omega_\varepsilon))$ and thus the equality $v_\varepsilon|_{t=0}=0$ makes sense.

Due to properties of the functions $k, k_i, \kappa_i, \ i=0, 1, 2,$ (see (\ref{cond-2})-(\ref{est-2'})) the operator ${\cal A}_\varepsilon$ is bounded,
strictly monotone, hemicontinuous, and coercive  (we verify these properties in more detail for the corresponding homogenized operator in Section \ref{hom-problem}). Then, from well-known results of the theory of monotone operators (see e.g. \cite{Showalter}) it follows that
for each fixed value $\varepsilon >0$ there exists a unique weak solution to problem (\ref{start_prob}).

\medskip

{\sf
Our main research efforts are oriented towards the analytical understanding and asymptotic approximation of phenomena and processes in physics and biology which take place in thick fractal junctions involving, as models, nonlinear boundary-value problem (\ref{start_prob}). In particular, we want  to find the corresponding homogenized problem as $\varepsilon \to 0,$  to construct the asymptotic approximation for the solution $v_\varepsilon$ and to study the influence of  the parameters $\{\alpha_i\}$ and $\{\beta_i\}$  on the asymptotic behavior of the solution.
}

\subsection{Features of the investigation}

\begin{enumerate}
  \item
Thick junctions have special character of the connectedness:
there are points in a thick junction, which are at a short distance of order ${\cal O}(\varepsilon),$ but the length
of all curves, which connect these points in the junction, is order ${\cal O}(1).$
As a result,  there are no extension operators that would be bounded uniformly in the corresponding Sobolev spaces \cite{ZAA99}.
At the same time the availability of an uniformly bounded family of extension operators is typical supposition in overwhelming majority of the existing homogenization schemes for problems in perforated domains with the Neumann or Robin boundary conditions
(see e.g. \cite{Ciorannescu,Conca}). In addition, thick junctions are non-convex domains with non-smooth boundaries. Therefore, solutions of boundary-value problems in such domains have only minimal $H^1$-smoothness, while (see e.g.~\cite{Conca}) the $H^2$-smoothness of a solution  is necessary to prove the convergence theorem. All these factors create special difficulties in the asymptotic analysis of BVP's in thick junctions.
  \item
In a typical interpretation the solution to problem (\ref{start_prob}) denotes the
density of some quantity (chemical concentration, temperature, electronic potential, etc) at equilibrium
within the thick fractal junction $\Omega_\varepsilon.$ Usually for applied problems, the source of the quantity is located in the junction’s body.
Therefore, the right-hand side $f_0$ is defined in $\Omega_0.$
  \item
Standard assumptions for nonlinear terms of reaction-diffusion equations are as follows:
they are Lipschtz continuous functions. This hypothesis in particular implies
$|k(s)| \le C (1 + |s|)$ for each $s \in \Bbb R$ and some constant $C.$ This is enough to state
that  problem~(\ref{start_prob}) has a unique solution. But, if we want to construct some approximation for a solution
and to prove the corresponding estimate, we need
some kind of a coercivity condition on the nonlinearity. Usually it reads as follows:
$k(s) s \ge C_1 |s|^2 - C_2$ for all $s \in \Bbb R$ and appropriate constants $C_1> 0,$
$C_2 \ge 0.$

Many physical processes, especially in chemistry and medicine, have monotonous nature. Therefore, it is naturally to
impose special monotonous conditions on the nonlinear terms. In our case we propose simple conditions (\ref{cond-2}) that imply
the coercivity conditions (\ref{est-1}).
  \item
Asymptotic behaviour of solutions to the reaction-diffusion equation in different kind of thin domains
with the uniform Neumann conditions was studied in \cite{ArCaPeSi,MarRyb}. The convergence theorems were proved under the following
assumptions for the nonlinear term:\\
in \cite{ArCaPeSi} it is a $C^2$-function with bounded derivatives and
\begin{equation}\label{disip-cond}
\limsup_{|s|\to+\infty} \frac{k(s)}{s} < 0;
\end{equation}
in \cite{MarRyb}  it is  a $C^1$-function, the dissipative condition (\ref{disip-cond}) holds and
\begin{equation}\label{deriv-cond}
|k'(s)| \le C (1 + |s|^{q-1}),
\end{equation}
where $q\in (1,+\infty)$.

Let us note that the convergence theorem for the solution to our problem~(\ref{start_prob})  can be proved under more weak assumptions for
 the functions $ k, \{k_i\}, \{\kappa_i\},$ namely they are vanish at zero and satisfy inequality (\ref{deriv-cond}).
  \item
 The nonlinear Robin boundary conditions are considered on the boundaries of the thin branches. These conditions mean that there is a flux of a  quantity through the surfaces of the branches. In fact very small activity holds always on the surface of
some material (therefore the Robin boundary conditions are more natural  for applied mathematical problems).
Such semilinear boundary conditions arise in many applied problems, in particular, in the modeling of chemical reactive flows. For instance, the following function
$$
\kappa(v) = \frac{\lambda\, v}{1 + \mu\, v} \quad \text{with} \quad \lambda, \mu > 0,
$$
which satisfies condition (\ref{cond-2}) if $f_0 \ge 0$ and $g_\varepsilon^{(i)} \equiv 0,$
corresponds to the Michaelis-Menten hypothesis in biochemical reactions
and to the Langmuir kinetics adsorption models (see \cite{Pao,Conca}).

\item
In the interpretation mentioned above, the problem (\ref{start_prob}) describes the motion of a reactive fluid having
different chemical features on different branching layers $(i=0, 1, 2)$ of the thick fractal junction.
To study the influence of the boundary interactions on the asymptotic behavior of the solution,
we introduce special intensity factors $\varepsilon^{\alpha_i}$ and $\varepsilon^{\beta_i}$ in
the Robin boundary conditions on the lateral sides of the thin rectangles from the $i$-th branching layers.

The effective behavior of this reactive flow  (as $\varepsilon \to 0)$ is described by a new nonstandard
homogenized parabolic problem containing extra zero-order terms which
catch the effect of the chemical reactions depending on $\alpha_i$ and $\beta_i.$
The asymptotic behavior of the solution is described in Theorem~\ref{th_1}. Here we note only that
 the following  differential equations
$$
h_{i,m} \partial_{t} v^{(i,m)}_0 -
h_{i,m} \partial^2_{x_2 x_2} v^{(i,m)}_0 + h_{i,m} k_i\big(v^{(i,m)}_0\big) +
 2 \delta_{\alpha_i,1} \kappa_i\big(v^{(i,m)}_0\big) =
 2 \delta_{\beta_i,1} g^{(i)}_0, \quad m = \overline{1, \, 2i},
$$
form the homogenized relations in  $D_i\times (0,T)$,  where
$\delta_{\alpha_i,1}, \delta_{\beta_i,1}$ are Kronecker's symbols.

\item
It should be stressed that the important problem for each new proposed asymptotic method is its
accuracy. Therefore, the proof of the error estimate for discrepancy between the constructed
approximation and the exact solution is general principle that has been applied to the
analysis of the efficiency of the proposed asymptotic method.
With the help of special branch-layer solutions and the method of matched asymptotic expansions, the approximation for the solution is constructed and the corresponding asymptotic error estimate in the space $C\big([0,T]; L^2(\Omega_\varepsilon) \big) \cap L^2\big(0, T; H^1(\Omega_\varepsilon)\big)$ is proved in Theorem~\ref{th_1}. From this theorem it follows directly the following corollary.
\begin{corollary}\label{corollary}
Let assumptions from Theorem~\ref{th_1} hold. Then for any $\rho\in(0,1)$
\begin{multline*}
\max_{t\in[0,T]}\Big( \|v_\varepsilon(\cdot,t) -
v^+_0(\cdot,t)\|_{L^2(\Omega_0)}
+ \sum_{i=0}^2 \sum_{m=1}^{2i}  \|v_\varepsilon(\cdot,t) -
v^{(i,m)}_0(\cdot,t)\|_{L^2(G^{(i,m)}_\varepsilon)}\Big)
\\
\le C_0 \Big(\varepsilon^{1-\rho}
 +  \sum_{i=0}^2\left(\varepsilon^{\alpha_i -1 + \delta_{\alpha_i, 1}} +
(1- \delta_{\beta_i, 1}) \varepsilon^{\beta_i - 1} + \delta_{\beta_i, 1}\|g^{(i)}_{\varepsilon } - g^{(i)}_{0}\|_{L^2(G^{(i)}_{\varepsilon})}
\right)\Big),
\end{multline*}
where $v_\varepsilon$ is the solution to problem {\rm(\ref{start_prob})}, $\Big(v^+, v^{(0)},  \big\{v^{(1,m)}\big\}_{m=1}^2, \, \big\{v^{(2,m)}\big\}_{m=1}^4\Big)$ is the multi-sheeted solution
to the homogenized problem~{\rm(\ref{CauchyProblem})}.
\end{corollary}

\end{enumerate}

\section{Formal asymptotic expansions for the solution}\label{Sec2}

\subsection{Outer expansions}\label{out_exp}

Combining the algorithm of constructing asymptotics in thin
domains with the methods of homogenization theory, we seek the
main terms of the asymptotics for the solution $v_\varepsilon$ in the form
\begin{equation}\label{3.1}
v_\varepsilon(x,t) \approx v^+_0(x,t)+\sum_{n=1}^{+\infty}\varepsilon^{n} v^+_n(x,t) \quad
\mbox{in domain} \  \Omega_0\times(0, T)
\end{equation}
and
\begin{equation}\label{3.2}
 v_\varepsilon(x,t) \approx
v^{(i,m)}_0(x,t) + \sum_{n=1}^{+\infty}\varepsilon^{n} v^{(i,m)}_n(x, \tfrac{x_1}{\varepsilon} - j,t)
\end{equation}
in the thin rod $G_j^{(i,m)}(\varepsilon)\times(0, T)$ from the $i$-th level, $j=0,\ldots,N-1.$ Let us recall that  $i \in \{0,1,2\}$
and the index $m\in \{1, 2\}$ for $i=1,$ \  $m\in \{1, 2, 3, 4\}$  for $i=2,$
and  if $i=0,$ then $m$ is absent and $G_j^{(0,m)}(\varepsilon)=G_j^{(0)}(\varepsilon)$ and $v^{(0,m)}_n =v^{(0)}_n.$

The asymptotic series (\ref{3.1}) and (\ref{3.2}) are usually called {\it outer expansions}.

Substituting the series (\ref{3.1}) in the first equation of problem
(\ref{start_prob}) and in the boundary conditions on
$\partial\Omega_0\setminus I_0,$ collecting coefficients of the
same powers of $\varepsilon$ and taking into account the first estimate in (\ref{est-2'}), we get the following relations for
the coefficient $v^+_0:$
\begin{equation}\label{3.3}
\begin{array}{rcll}
\partial_t v^+_0 - \Delta v^+_0 + k(v^+_0)& = & f_0& \quad  \text{in} \  \Omega_0\times(0, T),
\\[1mm]
\partial_{\nu} v^+_0 & = & 0 & \quad  \text{on} \  \big(\partial\Omega_0 \setminus I_0\big) \times(0, T).
\end{array}
\end{equation}

Now let us find limit relations in each domain $D_i$  (see (\ref{1.1})).
Assuming for the moment that the functions $\{v^{(i,m)}_n\}$ in (\ref{3.2}) are smooth, we write their Taylor series with respect to the variable $x_1$ at the point $x_1=\varepsilon(j+b_{i,m})$ (points $\{b_{i,m}\}$ are defined in (\ref{b1})--(\ref{b22}), $b_{0,m}=b_0= \frac12$) and pass to the "fast" \ variable $\xi_1=\varepsilon^{-1}x_1;$  \, the indexes $i,$ $m$ and $j$ are fixed. Then (\ref{3.2}) takes
the form
\begin{equation}\label{3.4}
v_\varepsilon(x,t) \approx v^{(i,m)}_0\bigl(\varepsilon(j+b_{i,m}),x_2,t\bigr)+ \sum_{n=1}^{+\infty}\varepsilon^{n}\,
   V^{(i,m,j)}_n(\xi_1,x_2,t),
\end{equation}
where
\begin{multline}
V^{(i,m,j)}_n(\xi_1,x_2,t) =
v^{(i,m)}_n\bigl(\varepsilon(j+b_{i,m}),x_2,\xi_1-j,t\bigr)
\\
+  \sum_{p=1}^{n}{( \xi_1 - j - b_{i,m})^p \over p!}\,
  {\partial^p v^{(i,m)}_{n-p} \over \partial x_1^p}
\bigl(\varepsilon(j+b_{i,m}),x_2,\xi_1-j,t\bigr). \label{3.5}
\end{multline}

Let us substitute (\ref{3.4}) into (\ref{start_prob}) instead of $v_{\varepsilon}.$ Since the Laplace operator takes the form
$\Delta=\varepsilon^{-2} {\partial^2 \over \partial\xi_1^2} + {\partial^2 \over \partial x_2^2},$
the collection of coefficients of the same power of $\varepsilon$ gives us one dimensional boundary value problems with respect to
$\xi_1$ for each $t\in (0,T).$ The first problem is the following:
\begin{equation}\label{3.6}
\begin{array}{rcll}
\partial^2_{\xi_1\xi_1} V^{(i,m,j)}_1(\xi_1,x_2,t) & = & 0,& \quad \xi_1\in I_{h_{i,m}}(b_{i,m}),
\\[1mm]
\partial_{\xi_1} V^{(i,m,j)}_1(b_{i,m} \pm \tfrac{h_{i,m}}{2},x_2,t) & = & 0,&
\end{array}
\end{equation}
where $\partial_{\xi_1}= \frac{\partial}{\partial \xi_1},$ $\partial^2_{\xi_1\xi_1}= \frac{\partial^2}{\partial \xi^2_1}$ and
$I_{h_{i,m}}(b_{i,m}) = \bigl( b_{i,m} - \frac{h_{i,m}}{2}\, , b_{i,m} +
\frac{h_{i,m}}{2} \bigr);$ \ the variable $x_2$ is regarded as a
parameter in this problem.

From (\ref{3.6}) it follows that function $V^{(i,m,j)}_1$ doesn't
depend on $\xi_1.$ Therefore, $V^{(i,m,j)}_1$ is equal to some
function $\varphi^{(i,m)}\bigl(\varepsilon(j+b_{i,m}),x_2,t\bigr).$ Since we look only for the first terms of the asymptotics, we can regard that
$\varphi^{(i,m)} \equiv 0.$ Then, due to
(\ref{3.5}), we have
\begin{equation}\label{3.7}
v^{(i,m)}_1\bigl(\varepsilon(j+b_{i,m}),x_2,\xi_1-j,t\bigr) = - \bigl( \xi_1 - j- b_{i,m} \bigr)\,   \partial_{x_1}
v^{(i,m)}_{0}\bigl(\varepsilon(j+b_{i,m}),x_2,t\bigr)\,.
\end{equation}

The problem for the function $V^{(i,m,j)}_2$ is as follows:
\begin{equation}\label{3.8}
-\partial^2_{\xi_1\xi_1} V^{(i,m,j)}_2 =
\Big(\partial^2_{x_2 x_2} v^{(i,m)}_0 - k_i\big(v^{(i,m)}_0\big)
- \partial_{t} v^{(i,m)}_0\Big)\Big|_{x_1=\varepsilon(j+b_{i,m})},
\ \   \xi_1\in  I_{h_{i,m}}(b_{i,m}),
\end{equation}
\begin{equation}
\partial_{\xi_1} V^{(i,m,j)}_2\bigl(\xi_1, x_2, t \bigr)\big|_{\xi_1=b_{i,m} \pm \tfrac{h_{i,m}}{2}} =
\Big(\mp \delta_{\alpha_i,1} \kappa_i\big(v^{(i,m)}_0(x,t)\big)
 \pm \delta_{\beta_i,1} g^{(i)}_0\bigl(x,t\bigr)\Big)\Big|_{x_1=\varepsilon(j+b_{i,m})},
\label{3.9}
\end{equation}
where $\delta_{\alpha_i,1}, \delta_{\beta_i,1}$ are Kronecker's symbols
(recall that $\alpha_i\ge 1$ and $\beta_i \ge1).$

The solvability condition for problem
(\ref{3.8})-(\ref{3.9}) is given by the differential
equation
\begin{equation}\label{3.10}
h_{i,m} \partial_{t} v^{(i,m)}_0 = h_{i,m} \partial^2_{x_2 x_2} v^{(i,m)}_0 -
h_{i,m} k_i\big(v^{(i,m)}_0\big) - 2 \delta_{\alpha_i,1} \kappa_i\big(v^{(i,m)}_0\big) +
 2 \delta_{\beta_i,1} g^{(i)}_0
\end{equation}
with respect to variables $x_2$ and $t$ at the fixed value of $x_1=\varepsilon(j+b_{i,m}).$

Since the points $\{x_1=\varepsilon(j+b_{i,m}):\ j=0,\ldots,N-1\} $ form the $\varepsilon$-net in
the interval $(0,a),$ we can extend all equations obtained above on $N$ segments to the
rectangle $D_i$ $(i=0, 1, 2).$ Thus, we get the following differential equation
\begin{equation}\label{3.10+1}
h_0 \partial_{t} v^{(0)}_0 =
h_0 \partial^2_{x_2 x_2} v^{(0)}_0 - h_0 k_0\big(v^{(0)}_0\big) -
 2 \delta_{\alpha_0,1} \kappa_0\big(v^{(0)}_0\big)  +  2 \delta_{\beta_0,1} g^{(0)}_0
\end{equation}
in $D_0 \times (0,T)$ $(h_{0,m}=h_0);$
 \ we get two differential equations (m=1, 2)
\begin{equation}\label{3.10+2}
h_{1,m} \partial_{t} v^{(1,m)}_0 =
h_{1,m} \partial^2_{x_2 x_2} v^{(1,m)}_0 -
h_{1,m} k_1\big(v^{(1,m)}_0\big) - 2 \delta_{\alpha_1,1} \kappa_1\big(v^{(1,m)}_0\big) +
 2 \delta_{\beta_1,1} g^{(1)}_0
\end{equation}
in $D_1\times (0,T);$ \ and we get four differential equations (m=1, 2, 3, 4)
\begin{equation}\label{3.10}
h_{2,m} \partial_{t} v^{(2,m)}_0 =
h_{2,m} \partial^2_{x_2 x_2} v^{(2,m)}_0 - h_{2,m} k_2\big(v^{(2,m)}_0\big) -
 2 \delta_{\alpha_2,1} \kappa_2\big(v^{(2,m)}_0\big) +
 2 \delta_{\beta_2,1} g^{(2)}_0
\end{equation}
in $D_2\times (0,T).$ Here the variable $x_1$ is regarded as a parameter.

If we substitute (\ref{3.4}) for $i=2$ into the Neumann condition on the bases
$$
Q^{(3)}_\varepsilon = \overline{\Omega}_\varepsilon \cap \{x: \ x_2 = - (l_1+l_2+l_3)\}
$$
and taking again that the points $\{x_1=\varepsilon(j+b_{2,m}):\ j=0,\ldots,N-1\} $ form the $\varepsilon$-net in
the interval $(0,a)$ in account, we obtain the following boundary conditions for functions $\{v^{(2,m)}_0\}:$
\begin{equation}\label{3.11}
\partial_{x_2} v^{(2,m)}_0\bigl(x_1, -(l_1+l_2+l_3),t\bigr) = 0, \quad  m=1, 2, 3, 4.
\end{equation}

To find transmission conditions on the joint zone $I_0$ and on each branching zones
$I_1=\{x: \ x_1\in (0, a), \ x_2= -l_1\},$ $I_2=\{x: \ x_1\in (0, a), \ x_2= -(l_1+l_2)\},$ we use the
method of matched asymptotic expansions for  the outer expansions
(\ref{3.1}), (\ref{3.2}) and inner expansions in neighborhoods of $I_0, I_1$ and $I_2.$

\subsection{Construction of inner expansions}\label{inn_exp}

\subsubsection{Inner expansion in a neighborhood of $I_0$}

\begin{wrapfigure}{r}{6cm}
\begin{center}
\includegraphics[height=6cm,width=4cm]{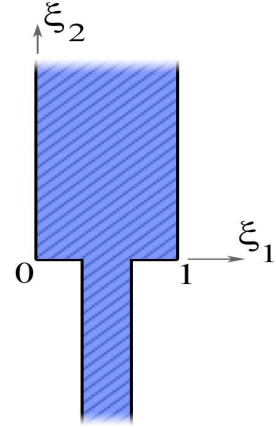}
\end{center}
\caption{The cell of periodicity $\Pi_0$}\label{F-cell1}
\end{wrapfigure}

 In a neighborhood of the joint
zone $I_0$ we introduce  the  "rapid" \ coordinates
$\xi=(\xi_1,\xi_2),$ where $\xi_1= \varepsilon^{-1}x_1$ and
$\xi_2=\varepsilon^{-1}x_2.$  Passing to $\varepsilon=0,$  we see
that the rod $G^{(0)}_0(\varepsilon)$ transforms into the semi-infinite strip
$$
\Pi^{-}_{h_0}= \big( \frac12 - \frac{h_0}{2} , \frac12 + \frac{h_0}{2} \big) \times (-\infty,0];
$$
the domain $\Omega_0$ transforms into the first quadrant  $\{\xi \, : \, \xi_1>0, \, \xi_2>0 \}.$
Taking into account the periodic structure of $\Omega_\varepsilon$ in a neighborhood of $I_0,$
we take the following cell of periodicity
$$
\Pi_0= \Pi^{-}_{h_0} \cup \Pi^+
$$
(see Fig.~\ref{F-cell1}), where junction-layer problems will be considered; here $\Pi^{+}=(0,1)\times(0,+\infty).$
Obviously, solutions of these joint-layer problems must be 1-periodic in
$\xi_1,$  i.e.,
\begin{equation}\label{periodic_cond}
\partial^{p}_{\xi_{1}}Z(\xi)\arrowvert_{\xi_{1}=0} =
\partial^{p}_{\xi_1}Z(\xi)\arrowvert_{\xi_{1}=1}\, ,
\qquad \xi \in \partial\Pi^{+} \, , \ \  \xi_2>0 , \ \ \ \  p=0, 1.
\end{equation}

We propose the following ansatz for the inner asymptotic expansion in a neighborhood of  $I_0 \cap \Omega_\varepsilon:$
\begin{equation} \label{3.12}
v_{\varepsilon} \approx  v^+_0(x_1,0,t) +  \varepsilon \Bigl(
Z^{(0)}_{1}\big(\tfrac{x}{\varepsilon}\big)\, \partial_{x_{1}} v^+_0(x_1,0,t) +
 Z^{(0)}_{2}\big(\tfrac{x}{\varepsilon}\big) \, \partial_{x_2}v^+_0(x_1,0,t) \Bigr) +\ldots
\end{equation}

Substituting (\ref{3.12}) in the differential equations of
problem (\ref{start_prob}) and in the corres\-ponding boundary
conditions, taking into account that the Laplace operator takes
the following form $\varepsilon^{-2} \Delta_\xi$ in the
coordinates $\xi$ and collecting the coefficients of the same
power of $\varepsilon,$ we arrive the following junction-layer problems for the
coefficients $Z^{(0)}_1$ and $Z^{(0)}_2$ (to these problems we must add the periodic conditions (\ref{periodic_cond})):
\begin{equation}\label{sol_Z_i}
\begin{array}{rcll}
-\Delta_{\xi} \ Z^{(0)}_p(\xi) &=& 0, &\quad \xi \in \Pi_0 ,
\\
\partial_{\xi_2}Z^{(0)}_p(\xi_1,0)&=& 0,&\quad
 \xi_1 \in (0,1)\setminus \big( \frac12 - \frac{h_0}{2} , \frac12 - \frac{h_0}{2} \big),
\\
\partial_{\xi_1} Z^{(0)}_p(\xi) &=& - \delta_{p,1},
& \quad \xi \in  \partial\Pi^{-}_{h_1} \cap \{\xi: \ \xi_2 <0\}, \qquad p=1, 2.
\end{array}
\end{equation}

The existence and the main asymptotic relations for solutions of problems~(\ref{sol_Z_i})
 can be obtained from general results about the asymptotic behavior of
solutions to elliptic problems in domains with different exits
to infinity \cite{Kond-Olei,Naz-Plam}.
However, if a domain, where we consider a boundary-value problem, has some symmetry,
then we can define more exactly the asymptotic relations and detect other properties of  junction-layer solutions
(see Lemma 4.1 and Corollary 4.1 from \cite{ZAA99}, see also \cite{MN96}). From those results  it follows
the following proposition.

\begin{prop}\label{prop1}
There exist unique solutions $Z^{(0)}_1, Z^{(0)}_2 \in H^1_{loc,\xi_2}(\Pi_0)$ to problems {\rm(\ref{sol_Z_i})}
respectively, which have the following differentiable asymptotics
\begin{equation}\label{w2}
Z^{(0)}_1(\xi)=
\begin{cases}
     \mathcal{O}(\exp(-2\pi \xi_2)), & \xi_2 \to +\infty, \\[3mm]
    \displaystyle{\Big(\!-\xi_1 +
    \frac12\Big)} + \mathcal{O}(\exp(\pi h^{-1}_0 \xi_2)), & \xi_2 \to -\infty,
\end{cases}
\end{equation}
\begin{equation}\label{w3}
Z^{(0)}_2(\xi)=
\begin{cases}
    \xi_2 + \mathcal{O}(\exp(-2\pi \xi_2)), & \xi_2 \to +\infty,
    \\[3mm]
    \displaystyle{\frac{\xi_2}{h_0}} +C_2 + \mathcal{O}(\exp(\pi h^{-1}_0 \xi_2)),
    & \xi_2 \to -\infty,
\end{cases}
\end{equation}
Moreover, function  $Z^{(0)}_1$ is odd in $\xi_1$ and  function
$Z^{(0)}_2$ is even in $\xi_1$ with respect to $\frac{1}{2}$.
\end{prop}

Recall  that  a function $Z$ belongs to the Sobolev space $ H^1_{loc,\xi_2}(\Pi_0)$ if for  every $R>0$ this function
$Z \in H^1(\Pi_0 \cap \{\xi: \, |\xi_2| < R\}).$

\subsubsection{Inner expansion in a neighborhood of the first branching zone $I_1$}

\begin{wrapfigure}[18]{r}{6cm}
\begin{center}
\includegraphics[height=6cm,width=4cm]{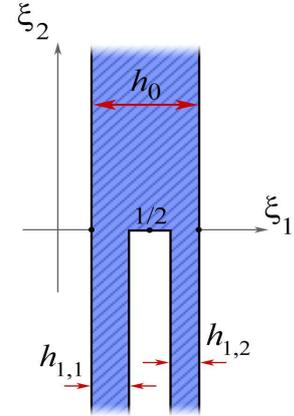}
\end{center}
\caption{The cell of periodicity $\Pi_0$}\label{F-cell2}
\end{wrapfigure}

In a neighborhood of $I_1$ we introduce  the  "rapid" \ coordinates
$\xi=(\xi_1,\xi_2),$ where $\xi_1= \varepsilon^{-1}x_1$ and
$\xi_2=\varepsilon^{-1}( x_2 + l_1).$  Passing to $\varepsilon=0$ , we see
that the rod $G^{(0)}_0(\varepsilon)$ transforms into the semi-infinite strip
$$
\Pi^+_{h_0} = \big( \frac12 - \frac{h_0}{2} , \frac12 + \frac{h_0}{2} \big) \times (0, +\infty)
$$
and rods $G^{(1,m)}_0(\varepsilon), \ m=1, 2,$ transform into the semi-infinite strips
$$
\Pi^-_{1,m} = \bigl( b_{1,m} - \frac{h_{1,m}}{2}\, , b_{1,m} + \frac{h_{1,m}}{2} \bigr) \times (-\infty, 0], \ m=1, 2,
$$
respectively.
Taking into account the periodic structure of $\Omega_\varepsilon$ in a neighborhood of $I_1,$
we take the following cell of periodicity
$$
\Pi_1 = \Pi^{+}_{h_0} \cup \Pi^-_{1,1} \cup \Pi^-_{1,2},
$$
where branch-layer problems will be considered.

Now we propose the following ansatz for the inner asymptotic expansion in a neighborhood of
$I_1\cap \big(G^{(0)}_\varepsilon \cup G^{(1)}_\varepsilon\big):$
\begin{multline}\label{3.12'}
v_{\varepsilon}(x,t) \approx  v^{(0)}_0(x_1,-l_1,t)  +  \varepsilon\Bigl(
Z^{(1)}_{1}\bigl(\tfrac{x_1}{\varepsilon} , \tfrac{x_2 + l_1}{\varepsilon}\bigr)  \partial_{x_{1}} v^{(0)}_0(x_1,-l_1,t)
\\
 + \bigl\{ \eta_1(x_1,t)\, \Xi^{(1)}_{1}\big(\tfrac{x_1}{\varepsilon} , \tfrac{x_2 + l_1}{\varepsilon}\big) +
(1-\eta_1(x_1,t) )  \, \Xi^{(1)}_{2}\big(\tfrac{x_1}{\varepsilon} , \tfrac{x_2 + l_1}{\varepsilon}\big) \bigr\}
\partial_{x_2}v^{(0)}_0(x_1,-l_1,t) \Bigr) +\ldots
\end{multline}
where  $Z^{(1)}_1(\xi), \, \Xi^{(1)}_1(\xi), \, \Xi^{(1)}_2(\xi),  \ \xi \in \Pi_1,$ \ are
solutions to branch-layer problems, which 1-periodic extended along the coordinate axis $O_{\xi_1},$ the function $\eta_1$ will be defined from matching conditions.

Substituting (\ref{3.12'}) in the corresponding differential equation of
problem (\ref{start_prob}) and boundary conditions,  we arrive branch-layer problems for the
functions $Z^{(1)}_1, \, \Xi^{(1)}_1, \, \Xi^{(1)}_2.$
So, the functions $\Xi^{(1)}_1$ and $\Xi^{(1)}_2$ are solution to the following homogeneous problem
\begin{equation}\label{3.13}
\begin{array}{rcll}
-\Delta_{\xi} \ \Xi(\xi) &=& 0, &\quad \xi \in \Pi_1 ,
\\
\partial_{\xi_1} \Xi(\xi) &=& 0,
& \quad \xi \in  \partial_{\parallel}\Pi_1,
\\
\partial_{\xi_2}\Xi(\xi_1,0)&=& 0,&\quad
 (\xi_1,0) \in \partial\Pi_1 \setminus \partial_{\parallel}\Pi_1,
\end{array}
\end{equation}
where $\partial_{\parallel}\Pi_1$ is the union of the vertical sides of $\partial\Pi_1.$
Again using approach mentioned above, we conclude.

\begin{prop}\label{prop1+1}
There exist two solutions $\Xi_1, \, \Xi_2 \in H^1_{loc,\xi_2}(\Pi_1)$ to problems
(\ref{3.13}), which have the following differentiable asymptotics:
\begin{equation}\label{3.14}
\Xi_1(\xi) =
\left\{
  \begin{array}{ll}
     \xi_2  + {\cal O}\big(\exp(- \frac{\pi \xi_2}{ h_0})\big), & \ \ \xi_2 \to +\infty,
    \   \xi \in \Pi^{+}_{h_0},
    \\[2mm]
    \dfrac{h_0}{h_{1,1}} \, \xi_2 + C_1^{(1)} + {\cal O}\big(\exp(\frac{\pi \xi_2}{ h_{1,1}})\big),
    & \ \ \xi_2 \to -\infty, \ \xi \in \Pi^{-}_{1,1},
\\[4mm]
   C_2^{(1)} + {\cal O}\big(\exp(\frac{\pi \xi_2}{ h_{1,2}})\big), & \ \
\xi_2 \to -\infty,  \ \xi \in \Pi^{-}_{1,2},
  \end{array}\right.
\end{equation}
\begin{equation}\label{3.15}
\Xi_2(\xi) =
\left\{
  \begin{array}{ll}
     \xi_2  + {\cal O}\big(\exp(- \frac{\pi \xi_2}{ h_0})\big), & \ \ \xi_2 \to +\infty,
    \   \xi \in \Pi^{+}_{h_0},
    \\[2mm]
   C_1^{(2)} + {\cal O}\big(\exp(\frac{\pi \xi_2}{ h_{1,1}})\big),
    & \ \ \xi_2 \to -\infty, \ \xi \in \Pi^{-}_{1,1},
\\[2mm]
   \dfrac{h_0}{h_{1,2}}\, \xi_2 + C_2^{(2)} + {\cal O}\big(\exp(\frac{\pi \xi_2}{ h_{1,2}})\big), & \ \
\xi_2 \to -\infty,  \ \xi \in \Pi^{-}_{1,2},
  \end{array}\right.
\end{equation}
where  $C_1^{(1)}, C_2^{(1)}, C_1^{(2)}, C_2^{(2)} $ are some fixed constants.

Any another solution to the homogeneous problem (\ref{3.13}),
which has polynomial grow at infinity, can be presented as a linear
combination $c_0 + c_1 \Xi_1 + c_2 \Xi_2.$
\end{prop}

The function $Z^{(1)}_1$ is a  solution to the following problem:
\begin{equation}\label{3.15'}
\begin{array}{rcll}
-\Delta_{\xi} \ Z(\xi) &=& 0, &\quad \xi \in \Pi_1 ,
\\
\partial_{\xi_1} Z(\xi) &=& -1,
& \quad \xi \in  \partial_{\parallel}\Pi_1,
\\
\partial_{\xi_2}Z(\xi_1,0)&=& 0,&\quad
 (\xi_1,0) \in \partial\Pi_1 \setminus \partial_{\parallel}\Pi_1.
\end{array}
\end{equation}

\begin{prop}\label{prop2}
There exists the unique solution $Z \in H^1_{loc,\xi_2}(\Pi_0)$ to problems {\rm(\ref{3.15'})},
 which has the following differentiable asymptotics:
\begin{equation}\label{3.16}
Z(\xi) =
\left\{
  \begin{array}{ll}
    -\xi_1 + \frac12  + {\cal O}\big(\exp(- \frac{\pi \xi_2}{ h_0})\big), & \ \ \xi_2 \to +\infty,
    \   \xi \in \Pi^{+}_{h_0},
    \\
     -\xi_1 + b_{1,1} + C_1 + {\cal O}\big(\exp(\frac{\pi \xi_2}{ h_{1,1}})\big),
    & \ \ \xi_2 \to -\infty, \ \xi \in \Pi^{-}_{1,1},
\\[2mm]
   -\xi_1 + b_{1,2} +  C_2 + {\cal O}\big(\exp(\frac{\pi \xi_2}{ h_{1,2}})\big), & \ \
\xi_2 \to -\infty,  \ \xi \in \Pi^{-}_{1,2},
  \end{array}\right.
\end{equation}
where  $C_1, C_2$ are some fixed constants.
\end{prop}

Thus, we set $\Xi^{(1)}_1=\Xi_1,$  $\Xi^{(1)}_2=\Xi_2$ and $Z^{(1)}_1= Z.$

\subsubsection{Inner expansion in a neighborhood of the second branching zone $I_2$}\label{zone-I_2}

 In a neighborhood of $I_2$ we introduce  the  "rapid" \ coordinates
$\xi=(\xi_1,\xi_2),$ where $\xi_1= \varepsilon^{-1}x_1$ and
$\xi_2=\varepsilon^{-1}( x_2 + l_1 +l_2).$  Passing to $\varepsilon=0,$  we see
that the rods $G^{(1,m)}_0(\varepsilon), \ m=1, 2,$ transform into the semi-infinite strips
$
\Pi^+_{1,m} = \bigl( b_{1,m} - \frac{h_{1,m}}{2}\, , b_{1,m} + \frac{h_{1,m}}{2} \bigr) \times (0, +\infty), \ m=1, 2,
$
respectively, and the rods $G^{(2,m)}_0(\varepsilon), \ m=1, 2, 3, 4,$ transform into the semi-infinite strips
$
\Pi^-_{2,m} = \bigl( b_{2,m} - \frac{h_{2,m}}{2}\, , b_{2,m} + \frac{h_{2,m}}{2} \bigr) \times (-\infty, 0], \ m=1, 2, 3, 4,
$
respectively.

Taking into account the periodic structure of $\Omega_\varepsilon$ in a neighborhood of $I_2,$
we take the following two cells of periodicity
$$
\Pi^{(1)}_2 = \Pi^{+}_{1,1} \cup \Pi^-_{2,1} \cup \Pi^-_{2,2}
\quad
\mbox{and}
\quad
\Pi^{(2)}_2 = \Pi^{+}_{1,2} \cup \Pi^-_{2,3} \cup \Pi^-_{2,4},
$$
where branch-layer problems will be considered.

Now we propose the following two inner asymptotic expansions in a neighborhood of
$I_2\cap \big(G^{(1)}_\varepsilon \cup G^{(2)}_\varepsilon\big),$ namely the first one is as follows:
\begin{multline}\label{in1}
 v_{\varepsilon}(x,t) \approx  v^{(1,1)}_0(x_1,0,t) +  \varepsilon\Bigl(
Z^{(2,1)}_{1}\bigl(\tfrac{x_1}{\varepsilon} , \tfrac{x_2 + l_1 + l_2}{\varepsilon}\bigr)\, \partial_{x_{1}} v^{(1,1)}_0(x_1,0,t)
\\
 + \Bigl\{ \eta_{2,1}(x_1,t)\, \Xi^{(2,1)}_{1}\big(\tfrac{x_1}{\varepsilon} , \tfrac{x_2 + l_1+ l_2}{\varepsilon}\big)
 +
(1-\eta_{2,1}(x_1,t) )  \, \Xi^{(2,1)}_{2}\big(\tfrac{x_1}{\varepsilon} , \tfrac{x_2 + l_1 + l_2}{\varepsilon}\big) \Bigr\}
\partial_{x_2}v^{(1,1)}_0(x_1,0,t) \Bigr) +\ldots
\end{multline}
in a neighborhood of $I_2\cap \Big( G^{(1,1)}_{\varepsilon} \bigcup \big(\bigcup_{m=1}^2  G^{(2,m)}_{\varepsilon}\big) \Big),$
and the second one
\begin{multline}\label{in2}
 v_{\varepsilon}(x,t) \approx  v^{(1,2)}_0(x_1,0,t) +  \varepsilon\Bigl(
Z^{(2,2)}_{1}\bigl(\tfrac{x_1}{\varepsilon} , \tfrac{x_2 + l_1 + l_2}{\varepsilon}\bigr)\, \partial_{x_{1}} v^{(1,2)}_0(x_1,0,t)
\\
+ \Bigl\{ \eta_{2,2}(x_1,t)\, \Xi^{(2,2)}_{1}\big(\tfrac{x_1}{\varepsilon} , \tfrac{x_2 + l_1 + l_2}{\varepsilon}\big)
+
(1-\eta_{2,2}(x_1,t) )  \, \Xi^{(2,2)}_{2}\big(\tfrac{x_1}{\varepsilon} , \tfrac{x_2 + l_1 + l_2}{\varepsilon}\big) \Bigr\}
\partial_{x_2}v^{(1,2)}_0(x_1,0,t) \Bigr) +\ldots
\end{multline}
in a neighborhood of
$I_2\cap \Big(G^{(1,2)}_{\varepsilon} \bigcup \big(\bigcup_{m=3}^4 G^{(2,m)}_{\varepsilon}\big) \Big).$

Coefficients $Z^{(2,1)}_{1}(\xi), \, \Xi^{(2,1)}_1(\xi), \, \Xi^{(2,1)}_2(\xi)  \ \big(\xi \in \Pi^{(1)}_2\big)$ in (\ref{in1})
and coefficients  $Z^{(2,2)}_{1}(\xi),$ $\Xi^{(2,2)}_1(\xi),$ $\Xi^{(2,2)}_2(\xi)  \ \big(\xi \in \Pi^{(2)}_2\big)$ in (\ref{in2})
are
solutions to branch-layer problems, which 1-periodic extended along
the coordinate axis $O_{\xi_1};$ the functions $\eta_{2,1}$ and $\eta_{2,2}$ will be defined from matching conditions.

Namely, $Z^{(2,1)}_{1}$ and $Z^{(2,2)}_{1}$ are solutions to problem (\ref{3.15'}) but now in  in $\Pi^{(1)}_2$
and $\Pi^{(2)}_2$ respectively. Applying  results of Proposition~\ref{prop2}, we can state that  there exist the unique solutions
 with the following differentiable asymptotics:
\begin{equation}\label{3.16'}
Z^{(2,1)}_{1}(\xi) =
\left\{
  \begin{array}{ll}
    -\xi_1 + b_{1,1}  + {\cal O}\big(\exp(- \frac{\pi \xi_2}{ h_{1,1}})\big), & \ \ \xi_2 \to +\infty,
    \   \xi \in \Pi^{+}_{1,1},
    \\
     -\xi_1 + b_{2,1} + C_1^{(3)} + {\cal O}\big(\exp(\frac{\pi \xi_2}{ h_{2,1}})\big),
    & \ \ \xi_2 \to -\infty, \ \xi \in \Pi^{-}_{2,1},
\\[2mm]
   -\xi_1 + b_{2,2} +  C_2^{(3)} + {\cal O}\big(\exp(\frac{\pi \xi_2}{ h_{2,2}})\big), & \ \
\xi_2 \to -\infty,  \ \xi \in \Pi^{-}_{2,2},
  \end{array}\right.
\end{equation}
\begin{equation}\label{3.16''}
Z^{(2,2)}_{1}(\xi) =
\left\{
  \begin{array}{ll}
    -\xi_1 + b_{1,2}  + {\cal O}\big(\exp(- \frac{\pi \xi_2}{ h_{1,2}})\big), & \ \ \xi_2 \to +\infty,
    \   \xi \in \Pi^{+}_{1,2},
    \\
     -\xi_1 + b_{2,3} + C_1^{(4)} + {\cal O}\big(\exp(\frac{\pi \xi_2}{ h_{2,3}})\big),
    & \ \ \xi_2 \to -\infty, \ \xi \in \Pi^{-}_{2,3},
\\[2mm]
   -\xi_1 + b_{2,4} +  C_2^{(4)} + {\cal O}\big(\exp(\frac{\pi \xi_2}{ h_{2,4}})\big), & \ \
\xi_2 \to -\infty,  \ \xi \in \Pi^{-}_{2,4}.
  \end{array}\right.
\end{equation}

Functions $\Xi^{(2,1)}_1,  \Xi^{(2,1)}_2$ and $\Xi^{(2,2)}_1,  \Xi^{(2,2)}_2$ are
 solutions to problem (\ref{3.13}) but now in $\Pi^{(1)}_2$ and $\Pi^{(2)}_2$ respectively. From Proposition~\ref{prop1+1} it follows
that they have the corresponding differentiable asymptotics (\ref{3.14}) and (\ref{3.15}).

\section{Matching of asymptotic expansions and homogenized problem}\label{Sec3}

We have formally constructed  the leading terms of the asymptotic expansions constructed in subsections~\ref{out_exp} and \ref{inn_exp}
 in different parts of the thick fractal junction $\Omega_\varepsilon.$
Next we apply the method of matched asymptotic expansions \cite{I} to complete the constructions.
Following this method,  the asymptotics of the leading terms of outer
expansions (\ref{3.1}) and (\ref{3.2}) as $x_2 \to \pm - \sum_{p=0}^m l_p, \ m=0, 1, 2,$ have to coincide with the corresponding asymptotics of
the inner expansions (\ref{3.12}), (\ref{3.12'}), (\ref{in1}) and (\ref{in2}) as $\eta_2\to \pm\infty$ respectively.

Near the point $(\varepsilon(j+ \frac12) , 0)\in I_0$ at the fixed
value of $t,$  the
function $v^{+}_0$ has the following asymptotics
$$
v^+_0(\varepsilon(j+\tfrac12),0,t) + \varepsilon\ \xi_2 \,
\partial_{x_2}v^+_0(\varepsilon(j+\tfrac12),0,t) + \ldots \quad \text{as} \  x_2 \to 0+0.
$$
Taking into account the asymptotics of  $Z^{(0)}_1$ and $Z^{(0)}_2$ as $\xi_2 \to +\infty$  (see (\ref{w2}) and  (\ref{w3})), we conclude that the matching conditions are  satisfied for the expansion (\ref{3.1}) and (\ref{3.12}).

The asymptotics of the outer expansion (\ref{3.2}) is equal to
\begin{equation}\label{3.17}
v^{(0)}_0(\varepsilon(j+ \tfrac12),0,t)
+  \varepsilon\Bigl(\bigl(-\xi_1+\tfrac12 +j\bigr)\, \partial_{x_1}v^{(0)}_0(\varepsilon(j+ \tfrac12),0,t)
 + \xi_2 \, \partial_{x_2}v^{(0)}_0(\varepsilon(j+\tfrac12),0,t) \Bigr)
 + \ldots
\end{equation}
as $x_2 \to 0-0, \quad  (x,t)\in G^{(0)}_j(\varepsilon)\times (0,T).$
Keeping in mind the asymptotics of functions $Z^{(0)}_1$ and $Z^{(0)}_2$ as $\xi_2\to-\infty,$ we find the asymptotics of the leading terms of  inner expansion (\ref{3.12})
\begin{equation}\label{3.18}
 v^+_0(\varepsilon(j+\tfrac12),0,t) +  \varepsilon\Bigl(
\bigl(-\xi_1+ j + \tfrac12 \bigr)\,
\partial_{x_1}v^+_0(\varepsilon(j+\tfrac12),0,t)
+ \bigl( \tfrac{\xi_2}{h_0} + C_2\bigr)\, \partial_{x_2}v^+_0(\varepsilon(j+\tfrac12),0,t) \Bigr) +\ldots
\end{equation}
 as $\  \xi_2 \to -\infty, \ \ \xi\in \Pi^{-}_{h_0}.$
 Comparing terms of (\ref{3.17}) and (\ref{3.18}) at $\varepsilon^0$ and $\varepsilon$ respectively, we
conclude that matching conditions are  satisfied if
$$
v^+_0(\varepsilon(j+ \tfrac12),0,t)=v^{(0)}_0(\varepsilon(j+\tfrac12),0,t),
\quad
\partial_{x_2}v^+_0(\varepsilon(j+\tfrac12),0,t) = h_0 \partial_{x_2}v^{(0)}_0(\varepsilon(j+\tfrac12),0,t),
$$
$j=0,1,\ldots,N-1.$ Since the points $\{x_1=\varepsilon(j+\tfrac12):\ j=0,\ldots,N-1\} $ form the $\varepsilon$-net in
the interval $(0,a),$ we can spread these relations into all interval $I_0$ and get the first transmission conditions
\begin{gather}\label{trans1'}
v^+_0(x_1,0,t)=v^{(0)}_0(x_1,0,t), \quad (x_1, t) \in (0,a) \times (0,T),
\\
\partial_{x_2}v^+_0(x_1,0,t) = h_0\, \partial_{x_2}v^{(0)}_0(x_1,0,t), \quad (x_1, t) \in (0,a) \times (0,T). \label{trans1''}
\end{gather}

Now we verify matching conditions at the point $(\varepsilon(j+ \frac12) , -l_1)\in I_1.$
It is easy to see that they are satisfied for the expansion (\ref{3.2}) as $x_2 \to -l_1+0$ $(x \in G^{(0)}_j(\varepsilon))$ and for the expansion
(\ref{3.12'}) as $\xi_2 \to +\infty$ $(\xi\in \Pi^{+}_{h_0}).$

Bearing in mind (\ref{3.14}), (\ref{3.15}) and (\ref{3.16}), we find at fixed value of $t \in (0,T)$ the following asymptotics of  (\ref{3.12'}):
\begin{gather}
v^{(0)}_0(\varepsilon(j+ b_{1,1}),-l_1,t) \,+ \, \varepsilon\Bigl(
\bigl(-\xi_1+ j + b_{1,1} + C_1\bigr)\,
\partial_{x_{1}} v^{(0)}_0(\varepsilon(j+ b_{1,1}),-l_1,t)  \notag
\\
+ \Bigl\{ \eta_1(\varepsilon(j+b_{1,1}),t) \, \bigl( \frac{h_0}{h_{1,1}}\, \xi_2 +
C^{(1)}_1\bigr)
+  \bigl( 1- \eta_1(\varepsilon(j+b_{1,1},t)) \bigr)
C^{(2)}_1 \Bigr\} \, \partial_{x_2}v^{(0)}_0(\varepsilon(j+ b_{1,1}),-l_1,t) \Bigr) +\ldots
 \notag \\
 \text{as} \ \ \xi_2 \to -\infty, \ \ \xi\in \Pi^{-}_{1,1},
 \label{3.18''}
 \end{gather}
and
\begin{gather}
v^{(0)}_0(\varepsilon(j+ b_{1,2}),-l_1,t) \,+ \, \varepsilon\Bigl(
\bigl(-\xi_1+ j + b_{1,2} + C_2\bigr)\,
\partial_{x_{1}} v^{(0)}_0(\varepsilon(j+ b_{1,2}),-l_1,t)  \notag
\\
+ \Bigl\{ \bigl( 1- \eta_1(\varepsilon(j+b_{1,2},t)) \bigr)
\bigl( \frac{h_0}{h_{1,2}}\, \xi_2 + C^{(2)}_2\bigr)
 + \eta_1(\varepsilon(j+b_{1,2}),t) \, C^{(1)}_2 \Bigr\} \, \partial_{x_2}v^{(0)}_0(\varepsilon(j+ b_{1,2}),-l_1,t) \Bigr) +\ldots
 \notag
 \\
\text{as} \ \ \xi_2 \to -\infty, \ \ \xi\in \Pi^{-}_{1,2}.
 \label{3.18'''}
\end{gather}

Asymptotic forms of outer expansions (\ref{3.2}) at $i=1$ and $m=1, 2$ are equal to
\begin{gather}
v^{(1,1)}_0(\varepsilon(j+ b_{1,1}),-l_1,t)  +
 \varepsilon\Bigl(\bigl(-\xi_1+ b_{1,1} +j\bigr)\, \partial_{x_1}v^{(1,1)}_0(\varepsilon(j+ b_{1,1}),-l_1,t)
 \notag
\\
 + \, \xi_2 \, \partial_{x_2}v^{(1,1)}_0(\varepsilon(j+b_{1,1}),-l_1,t) \Bigr) + \ldots  \label{out1_3.17}
\end{gather}
as $x_2 \to -l_1-0, \  x \in G^{(1,1)}_j(\varepsilon),$ and
\begin{gather}
v^{(1,2)}_0(\varepsilon(j+ b_{1,2}),-l_1,t)  +
 \varepsilon\Bigl(\bigl(-\xi_1+ b_{1,2} +j\bigr)\, \partial_{x_1}v^{(1,2)}_0(\varepsilon(j+ b_{1,2}),-l_1,t)
 \notag
\\
 +  \xi_2 \, \partial_{x_2}v^{(1,2)}_0(\varepsilon(j+b_{1,2}),-l_1,t) \Bigr)
 + \ldots  \label{out2_3.17}
\end{gather}
as $ x_2 \to -l_1-0, \  x \in G^{(1,2)}_j(\varepsilon).$

To satisfy the matching conditions, we compare terms of (\ref{3.18''}) and  (\ref{out1_3.17}), (\ref{3.18'''}) and (\ref{out2_3.17}) at
$\varepsilon^0$ and $\varepsilon^1.$ As a result, we get
\begin{gather*}\label{3.20}
v^{(0)}_0(\varepsilon(j+ b_{1,m}),-l_1,t) =v^{(1,m)}_0(\varepsilon(j+ b_{1,m}),-l_1,t) , \quad m=1, 2,
\\[2mm]
 \eta_1(\varepsilon(j+b_{1,1}),t) \,  h_0\,\partial_{x_2}v^{(0)}_0(\varepsilon(j+ b_{1,1}),-l_1,t) = h_{1,1}\,
\partial_{x_2}v^{(1,1)}_0(\varepsilon(j+b_{1,1}),-l_1,t), \label{3.21}
\\[2mm]
\bigl( 1- \eta_1(\varepsilon(j+b_{1,2},t)) \bigr)\, h_0\, \partial_{x_2}v^{(0)}_0(\varepsilon(j+ b_{1,2}),-l_1,t)=h_{1,2}\,
\partial_{x_2}v^{(1,2)}_0(\varepsilon(j+b_{1,2}),-l_1,t), \label{3.22}
\end{gather*}
for $j=0,1,\ldots,N-1.$
Since the sets $\{x_1=\varepsilon(j+b_{1,1}):\ j=0,\ldots,N-1\} $ $\{x_1=\varepsilon(j+b_{1,2}):\ j=0,\ldots,N-1\} $
form the $\varepsilon$-net in the interval $(0,a),$ we can spread these relations
into all interval $I_1$ and deduce the second transmission conditions
\begin{gather}\label{second_trans1}
v^{(0)}_0(x_1,-l_1,t) = v^{(1,m)}_0(x_1,-l_1,t), \quad m=1, 2,
\\
 h_0 \partial_{x_2}v^{(0)}_0(x_1,-l_1,t) = h_{1,1} \partial_{x_2}v^{(1,1)}_0(x_1,-l_1,t) + h_{1,2} \partial_{x_2}v^{(1,2)}_0(x_1,-l_1,t) \label{second_trans2}
\end{gather}
and determine the function
\begin{equation}\label{3.23}
  \eta_1(x_1,t) := \frac{h_{1,1}\, \partial_{x_2}v^{(1,1)}_0(x_1,-l_1,t)}
  {h_{1,1}\, \partial_{x_2}v^{(1,1)}_0(x_1,-l_1,t) +  h_{1,2} \,\partial_{x_2}v^{(1,2)}_0(x_1,-l_1,t)}
\end{equation}
for $ x_1 \in (0, a)$ and $t\in (0,T).$

Due to (\ref{second_trans1})
$$
\bigl(-\xi_1+ j + b_{1,m} \bigr)\, \partial_{x_{1}} v^{(0)}_0(\varepsilon(j+ b_{1,m}),-l_1,t)
= \bigl(-\xi_1+ j + b_{1,m}\bigr)\, \partial_{x_1}v^{(1,m)}_0(\varepsilon(j+ b_{1,m}),-l_1,t), \quad m=1, 2.
$$
Therefore, the matching conditions are satisfied for the leading terms of asymptotic expansions (\ref{3.2}) and (\ref{3.12'})
at each point $(\varepsilon(j+ \frac12) , -l_1)\in I_1,$ $j=0,1,\ldots,N-1,$  if
(\ref{second_trans1}), (\ref{second_trans2}) and (\ref{3.23}) hold.

In analogous way we can deduce the following two kinds of transmission conditions at $x_2=-(l_1+l_2):$
\begin{gather}
v^{(1,1)}_0 = v^{(2,1)}_0 = v^{(2,2)}_0  \quad \text{on} \ I_2\times(0, T), \label{third_trans1}
\\
 h_{1,1} \partial_{x_2}v^{(1,1)}_0 = h_{2,1} \partial_{x_2}v^{(2,1)}_0 + h_{2,2} \partial_{x_2}v^{(2,2)}_0
 \quad \text{on} \ I_2\times(0, T),\label{third_trans2}
\end{gather}
and
\begin{gather}
v^{(1,2)}_0 = v^{(2,3)}_0 = v^{(2,4)}_0  \quad \text{on} \ I_2\times(0, T), \label{third_trans3}
\\
 h_{1,2} \partial_{x_2}v^{(1,2)}_0 = h_{2,3} \partial_{x_2}v^{(2,3)}_0 + h_{2,4} \partial_{x_2}v^{(2,4)}_0
 \quad \text{on} \ I_2\times(0, T). \label{third_trans4}
\end{gather}
In addition, the functions $\eta_{2,1}$ and $\eta_{2,2}$ in (\ref{in1}) and (\ref{in2}) are defined by formulas
\begin{equation}\label{3.24}
  \eta_{2,1}(x_1,t)= \frac{h_{2,1}\, \partial_{x_2}v^{(2,1)}_0(x_1,-(l_1+l_2),t)}
  {h_{2,1} \,\partial_{x_2}v^{(2,1)}_0(x_1,-(l_1+l_2),t) +  h_{2,2} \,\partial_{x_2}v^{(2,2)}_0(x_1,-(l_1+l_2),t)},
\end{equation}
\begin{equation}\label{3.25}
  \eta_{2,2}(x_1,t)= \frac{h_{2,3}\, \partial_{x_2}v^{(2,3)}_0(x_1,-(l_1+l_2),t)}
  {h_{2,3}\, \partial_{x_2}v^{(2,3)}_0(x_1,-(l_1+l_2),t) +  h_{2,4} \, \partial_{x_2}v^{(2,4)}_0(x_1,-(l_1+l_2),t)}.
\end{equation}

Relations (\ref{3.3}), (\ref{3.10+1})-(\ref{3.11}), (\ref{trans1'}), (\ref{trans1''}), (\ref{second_trans1}), (\ref{second_trans2}),
(\ref{third_trans1})-(\ref{third_trans4}) form  {\it homogenized problem} for problem (\ref{start_prob}).

\section{Operator formulation of the homogenized problem}\label{hom-problem}

To give appropriately the following definition of a weak solution of the homogenized problem, let us first introduce an
 anizotropic Sobolev space ${\bf H}$ of multi-sheeted functions.
 A multi-sheeted function
$$
\boldsymbol{\varphi}:=
\Big(\varphi^+, \varphi^{(0)},  \big\{\varphi^{(1,m)}\big\}_{m=1}^2, \, \big\{\varphi^{(2,m)}\big\}_{m=1}^4\Big) =
\left\{
\begin{array}{ll}
                     \varphi^+(x), & x \in \ \Omega_0,
                      \\
                      \varphi^{(0)}(x), & x \in  \ D_0,
                      \\
                      \varphi^{(1,m)}(x), & x \in  \ D_1, \quad m =1, 2,
                      \\
                      \varphi^{(2,m)}(x), & x \in  \ D_2, \quad m =1, 2, 3, 4,
                    \end{array}
                  \right.
$$
belongs to ${\bf H}$ if
$\varphi^+ \in H^1(\Omega_0),$ \ $\{\varphi^{(i,m)}\}_{m=1}^{2i} \subset L^2(D_i),$  there exist weak derivatives
$\{\partial_{x_2} \varphi^{(i,m)}\}_{m=1}^{2i} \subset L^2(D_i), \ i=0, 1, 2,$ and
$$
\varphi^+|_{I_0} =\varphi^{(0)}|_{I_0}, \qquad \varphi^{(0)}|_{I_1} = \varphi^{(1,1)}|_{I_1}=\varphi^{(1,2)}|_{I_1},
$$
$$
\varphi^{(1,1)}|_{I_2}= \varphi^{(2,1)}|_{I_2}=\varphi^{(2,2)}|_{I_2}, \qquad
\varphi^{(1,2)}|_{I_2}= \varphi^{(2,3)}|_{I_2}=\varphi^{(2,4)}|_{I_2}.
$$
Obviously, the space ${\bf H}$ is continuously and densely embedded in the Hilbert space ${\bf V}$ of multi-sheeted functions
whose components belong to the corresponding $L^2$-spaces, i.e.,
$\boldsymbol{\varphi} \in {\bf V}$ if $\varphi^+ \in L^2(\Omega_0),$ $\{\varphi^{(i,m)}\}_{m=1}^{2i} \subset L^2(D_i), \ i=0, 1, 2.$
The scalar products in ${\bf V}$ and ${\bf H}$ are defined as follows:
$$
\boldsymbol{(} \boldsymbol{\varphi}, \boldsymbol{\psi} \boldsymbol{)}_{{\bf V}}:= (\varphi^+ , \psi^+)_{L^2(\Omega_0)} +
+ \sum_{i=0}^2\sum_{m=1}^{2i}(\varphi^{(i,m)}, \psi^{(i,m)})_{L^2(D_i)},
$$
$$
\boldsymbol{(} \boldsymbol{\varphi}, \boldsymbol{\psi} \boldsymbol{)}_{{\bf H}}:=
\boldsymbol{(} \boldsymbol{\varphi}, \boldsymbol{\psi} \boldsymbol{)}_{{\bf V}} +
(\nabla\varphi^+ , \nabla\psi^+)_{L^2(\Omega_0)} +
\sum_{i=0}^2\sum_{m=1}^{2i}(\partial_{x_2}\varphi^{(i,m)}, \partial_{x_2}\psi^{(i,m)})_{L^2(D_i)}
$$
Recall that $\varphi^{(0,m)}=\varphi^{(0)}$ (see Remark~\ref{remark_1}).
Since ${\bf H}$ is continuously and densely embedded in ${\bf V},$  we can construct the Gelfand triple ${\bf H} \subset {\bf V} \subset {\bf H}^*.$

For almost every  $t\in [0,T]$ we introduce an operator $\boldsymbol{{\cal A}}(t): {\bf H} \mapsto {\bf H}^*$  by the formula
$$
\boldsymbol{\langle} \boldsymbol{{\cal A}}(t)\boldsymbol{\varphi}, \boldsymbol{\psi} \boldsymbol{\rangle}:= \int\limits_{\Omega_0} \Big(\nabla \varphi^+ \cdot \nabla \psi^+ +
k(\varphi^+) \psi^+\Big)\, dx
$$
$$
+ \sum_{i=0}^2\sum_{m=1}^{2i}
\int\limits_{D_i}\Big(h_{i,m} \partial_{x_2}\varphi^{(i,m)} \, \partial_{x_2}\psi^{(i,m)} + h_{i,m} k_i(\varphi^{(i,m)}) \psi^{(i,m)} +
2 \delta_{\alpha_i,1} \kappa_i(\varphi^{(i,m)}) \psi^{(i,m)}\Big)dx
$$
for all $\boldsymbol{\varphi},  \boldsymbol{\psi} \in L^2(0,T; {\bf H}),$ and a linear functional ${\bf F}(t) \in {\bf H}^*$
$$
\boldsymbol{\langle} {\bf F}(t) , \boldsymbol{\psi} \boldsymbol{\rangle}:= \int\limits_{\Omega_0} f_0 \psi^+ \, dx + 2
\sum_{i=0}^2\delta_{\beta_i,1} \sum_{m=1}^{2i}  \int\limits_{D_i} g_0^{(i)}  \psi^{(i,m)}  dx.
 $$
Here $\boldsymbol{ \langle \cdot , \cdot \rangle }$ is the pairing of ${\bf H}^*$ and ${\bf H},$ \ $\boldsymbol{\psi} =
\Big(\psi^+, \psi^{(0)},  \big\{\psi^{(1,m)}\big\}_{m=1}^2, \, \big\{\psi^{(2,m)}\big\}_{m=1}^4\Big)\in L^2(0,T; {\bf H}).$

Now we can write down the homogenized problem in the form of the abstract Cauchy problem
\begin{equation}\label{CauchyProblem}
{\bf v}^{\prime} + \boldsymbol{{\cal A}}({\bf v}) = {\bf F}
\ \ \ \textrm{in} \ \ L^2\bigl(0,T; {\bf H}^*\bigr), \quad \
{\bf v}(0)=0,
\end{equation}
where ${\bf v}= \Big(v^+, v^{(0)},  \big\{v^{(1,m)}\big\}_{m=1}^2, \, \big\{v^{(2,m)}\big\}_{m=1}^4\Big)\in L^2(0,T; {\bf H}).$

\begin{definition}\label{def-weak-solution}
We say a multi-sheeted function
$$
{\bf v} \in L^2(0,T; {\bf H}), \quad \mbox{with} \quad {\bf v}^\prime \in L^2(0,T; {\bf H}^*),
$$
is a weak solution to the homogenized problem provided
$$
\boldsymbol{ \langle} {\bf v}^\prime(t), {\bf u} \boldsymbol{\rangle } + \boldsymbol{\langle} \boldsymbol{{\cal A}}(t){\bf v}, {\bf u} \boldsymbol{\rangle}
= \boldsymbol{\langle} {\bf F}(t) , {\bf u} \boldsymbol{\rangle} \quad \forall \ {\bf u} \in {\bf H} \ \ \mbox{and for a.e.} \ \ t \in (0,T),
$$
and ${\bf v}|_{t=0}=0.$
\end{definition}

\begin{remark}
In view of the well-known properties of spaces $L^p(0,T; X)$ (see for instance \cite{Showalter}), the weak solution
 ${\bf v} \in C([0,T]; {\bf V}),$ and thus the last equality in Definition~\ref{def-weak-solution} makes sence.
\end{remark}

\begin{theorem} There exists a unique weak multi-sheeted solution
 to the homogenized problem.
\end{theorem}

\begin{proof} Let us show that for a.e. $t\in (0,T)$ the operator $\boldsymbol{{\cal A}}$  is  bounded, strictly monotone, and hemicontinuous.

(1)
Using (\ref{est-1}), (\ref{est-2})  and the definition of $\boldsymbol{{\cal A}},$  we can prove the following inequality
$$
\big|\boldsymbol{\langle} \boldsymbol{{\cal A}}(t)\boldsymbol{\varphi}, \boldsymbol{\psi} \boldsymbol{\rangle}\big| \le C_1 (1 + \|\boldsymbol{\varphi}\|_{{\bf H}}) \|\boldsymbol{\psi}\|_{{\bf H}} \quad \forall \, \boldsymbol{\varphi}, \boldsymbol{\psi} \in {\bf H},
$$
from where it follow that $\boldsymbol{{\cal A}}$ is bounded.

(2) Operator $\mathcal{A}$ is strongly monotone. Really, with the help of (\ref{cond-2}) we get
$$
\boldsymbol{\langle} \boldsymbol{{\cal A}}\boldsymbol{\varphi} - \boldsymbol{{\cal A}}\boldsymbol{\psi} , \boldsymbol{\varphi} - \boldsymbol{\psi} \boldsymbol{\rangle} \ge \int\limits_{\Omega_0} \Big(\big|\nabla (\varphi^+ - \psi^+)\big|^2 + c_1 \big|\varphi^+ - \psi^+\big|^2 \Big)\, dx
$$
$$
+ \sum_{i=0}^2\sum_{m=1}^{2i}
\int\limits_{D_i}\Big(h_{i,m} \big| \partial_{x_2}\varphi^{(i,m)} - \partial_{x_2}\psi^{(i,m)}\big|^2 +
(h_{i,m} + 2 \delta_{\alpha_i,1}) c_1 \big|\varphi^{(i,m)} - \psi^{(i,m)}\big|^2 \Big)\, dx
$$
$$
\geq C_2 \|\boldsymbol{\varphi} - \boldsymbol{\psi}\|^2_{{\bf H}} \qquad \forall \ \boldsymbol{\varphi},  \boldsymbol{\psi} \in {\bf H}.
$$

(3) Operator $\boldsymbol{{\cal A}}$ is hemicontinuous. Indeed,   the real valued function
$$
[0, 1] \ni \tau \longmapsto  \boldsymbol{\langle} \boldsymbol{{\cal A}}(\boldsymbol{\varphi} + \tau \boldsymbol{\upsilon}) , \boldsymbol{\psi} \boldsymbol{\rangle}
$$
is continuous on $[0,1]$ for all fixed $\boldsymbol{\varphi},  \boldsymbol{\psi}, \boldsymbol{\upsilon} \in {\bf H}$
due to the continuity of the functions $\{k, k_i, \kappa_i\},$  the right inequality in (\ref{est-2'}), and Lebesgue's dominated convergence theorem.

Thus,  the realization $\boldsymbol{{\cal A}}: L^2\bigl(0,T; {\bf H}\bigr)
\mapsto L^2\bigl(0,T; {\bf H}^*\bigr)$ (we denote it by the same symbol)
is bounded, monotone, and hemicontinuous, i.e.,  $\boldsymbol{{\cal A}}$ is type of $M$ (see Lemma~2.1 \cite{Showalter}).

(4) Operator $\mathcal{A}$ is coercive. Using (\ref{est-1}), (\ref{est-2'}), and the Cauchy's inequality with $\delta$ $(ab\le\delta a^{2}+\frac{b^{2}}{4\delta},\ \ a, b, \delta>0),$ we find
$$
\int_{0}^{T} \boldsymbol{\langle} \boldsymbol{{\cal A}}(t)\boldsymbol{\varphi}, \boldsymbol{\varphi} \boldsymbol{\rangle}  \,dt \ge
C_3 \int_{0}^{T}\|\boldsymbol{\varphi}\|^2_{{\bf H}}\, dt  - |k(0)| \int\limits_{0}^{T}\int\limits_{\Omega_0} |\varphi^+| \, dxdt
$$
$$
- \sum_{i=0}^2\sum_{m=1}^{2i}
\big(h_{i,m}  |k_i(0)|  + 2 \delta_{\alpha_i,1} |\kappa_i(0)|\big) \int\limits_{0}^{T}\int\limits_{D_i} |\varphi^{(i,m)}| \, dxdt
$$
$$
\ge C_3 \int_{0}^{T}\|\boldsymbol{\varphi}\|^2_{{\bf H}}\, dt - \delta  \int_{0}^{T}\|\boldsymbol{\varphi}\|^2_{{\bf V}}\, dt - C_4(\delta)
$$
for each $\boldsymbol{\varphi} \in L^2\bigl(0,T; {\bf H}\bigr).$ By selecting appropriate $\delta,$ we obtain the desired inequality
for the coerciveness.

By Corollary 4.1 \cite{Showalter},  problem (\ref{CauchyProblem}) has a
unique solution.
\end{proof}

\section{Asymptotic approximation}\label{approx}

Let ${\bf v} = \Big(v^+, v^{(0)},  \big\{v^{(1,m)}\big\}_{m=1}^2, \, \big\{v^{(2,m)}\big\}_{m=1}^4\Big) \in  L^2(0,T; {\bf H})$ be a unique weak solution to the homogenized problem (\ref{CauchyProblem}).
With the help of ${\bf v},$ the junction-layer solutions  $Z^{(0)}_1$ and $Z^{(0)}_2$ (see Proposition~\ref{prop1}),
the branch-layer solutions $\big\{Z^{(1)}_1, \Xi^{(1)}_1,  \Xi^{(1)}_2\big\}$ (see Propositions~\ref{prop1+1}, \ref{prop2})
in a neighborhood of the first branching zone $I_1$,
and the branch-layer solutions $\big\{Z^{(2,1)}_{1}, Z^{(2,2)}_{1},$ $ \Xi^{(2,1)}_1,  \Xi^{(2,1)}_2, \Xi^{(2,2)}_1,  \Xi^{(2,2)}_2\big\}$
in a neighborhood of the second branching zone $I_2$ (see \S~\ref{zone-I_2}), we  define the leading terms in
the asymptotic expansions  (\ref{3.1}), (\ref{3.2}), (\ref{3.12}),  (\ref{3.12'}), (\ref{in1}), and (\ref{in2}).

An approximating function $R_\varepsilon$ is constructed as the sum of the leading terms
of the outer expansions (\ref{3.1}), (\ref{3.2}) and the inner expansion (\ref{3.12}),  (\ref{3.12'}), (\ref{in1}),  (\ref{in2})
in neighborhoods of the joint zone $I_0$ and branching zones $I_1, I_2$ respectively, with the subtraction of the identical terms of their asymptotics because they are summed twice.
As a result, we obtain
\begin{equation}
R_{\varepsilon}(x,t) = v^{+}(x,t) + \varepsilon
\chi_0(x_2)\, {\cal N}^{(0)}_{+}\bigl(\tfrac{x}{\varepsilon}, x_1 ,t \bigr), \quad  (x,t) \in \Omega_0 \times (0,T); \label{4.1}
\end{equation}
\begin{multline}
R_{\varepsilon} = v^{(0)}(x,t)  + \varepsilon\Bigl( Y_0(\tfrac{x_1}{\varepsilon})\, \partial_{x_1}v^{(0)}(x,t) +
\chi_0(x_2){\cal N}^{(0)}_{-}\bigl(\tfrac{x}{\varepsilon}, x_1, t \bigr)
+ \chi_1(x_2){\cal N}^{(1)}\bigl(\tfrac{x_1}{\varepsilon}, \tfrac{x_2 + l_1}{\varepsilon}, x_1, t \bigr) \Bigr),
\\
 (x,t) \in  G^{(0)}_\varepsilon\times (0,T);  \label{4.2}
\end{multline}
\begin{gather}
R_{\varepsilon} = v^{(1,m)}(x,t)  + \varepsilon\Bigl( Y_{1,m}(\tfrac{x_1}{\varepsilon})\, \partial_{x_1}v^{(1,m)}(x,t) +
\chi_1(x_2){\cal N}^{(1)}_{1,m}\bigl(\tfrac{x_1}{\varepsilon}, \tfrac{x_2 + l_1}{\varepsilon}, x_1, t \bigr)
\notag
\\[2mm]
+ \chi_2(x_2){\cal N}^{(2)}_m\bigl(\tfrac{x_1}{\varepsilon}, \tfrac{x_2 + l_1 + l_2}{\varepsilon}, x_1, t \bigr) \Bigr), \quad (x,t) \in  G^{(1,m)}_\varepsilon\times (0,T), \ m = 1, 2;  \label{4.3}
\end{gather}
\begin{gather}
R_{\varepsilon} = v^{(2,m)}(x,t)  + \varepsilon\Bigl( Y_{2,m}(\tfrac{x_1}{\varepsilon})\, \partial_{x_1}v^{(2,m)}(x,t) +
\chi_2(x_2){\cal N}^{(2)}_{2,m}\bigl(\tfrac{x_1}{\varepsilon}, \tfrac{x_2 + l_1 + l_2}{\varepsilon}, x_1, t \bigr) \Bigr),
\notag
\\[2mm]
 (x,t) \in  G^{(2,m)}_\varepsilon\times (0,T), \quad m = 1, 2, 3, 4.  \label{4.4}
\end{gather}
Here the function $\chi_0$ is a smooth cutoff function such that $\chi_0(x_2)=1$ for $|x_2| \le \tau_0/2,$ and
$\chi_0(x_2)=0$ for $|x_2| \ge \tau_0,$ where $\tau_0$ is sufficiently small number; $\chi_1(x_2)= \chi_0(x_2 + l_1),$
$\chi_2(x_2)= \chi_0(x_2 + l_1 + l_2),$ $x_2 \in \Bbb R;$\\
in {\bf (\ref{4.1})}
$$
{\cal N}^{(0)}_{+}\bigl(\xi, x_1 ,t \bigr)= \sum_{i=1}^{2}\bigl(Z^{(0)}_i(\xi) - \delta_{i,2}\xi_2\bigr) \partial_{x_i}v^{+}(x_1,0,t), \quad \xi = \tfrac{x}{\varepsilon},
$$
where $\delta_{i,2}$ is the Kronecker delta; \\
in {\bf (\ref{4.2})} $Y_0(\xi_1)=-\xi_1 + \frac12 + [\xi_1],$ where $[\xi_1]$ is the
entire part of $\xi_1,$ and
$$
{\cal N}^{(0)}_{-} = \Bigl(   Z^{(0)}_1(\xi) - Y_0(\xi_1) \Bigr) \partial_{x_1}v^{+}(x_1,0,t)
+  \Big( Z^{(0)}_2(\xi) - \tfrac{\xi_2}{h_0} \Big) \partial_{x_2}v^{+}(x_1,0,t), \quad \xi = \tfrac{x}{\varepsilon},
 $$
 \begin{multline*}
{\cal N}^{(1)} = \Big(Z^{(1)}_{1}(\xi) -  Y_0(\xi_1) \Bigr) \partial_{x_{1}} v^{(0)}(x_1,-l_1,t)
\\
+
 \Bigl( \eta_1(x_1,t) \Xi^{(1)}_{1}(\xi)
 + (1-\eta_1(x_1,t) ) \Xi^{(1)}_{2}(\xi) - \xi_2\Bigr) \partial_{x_2}v^{(0)}(x_1,-l_1,t),
 \\ \xi_1 = \tfrac{x_1}{\varepsilon}, \  \xi_2 = \tfrac{x_2 + l_1}{\varepsilon};
 \end{multline*}
in {\bf (\ref{4.3})}  $Y_{1,m}(\xi_1)= -\xi_1 + b_{1,m} + [\xi_1], \ m=1, 2,$ and
 \begin{multline*}
{\cal N}^{(1)}_{1,m}\bigl(\xi, x_1, t \bigr) = \Big(Z^{(1)}_{1}(\xi) -  Y_{1,m}(\xi_1) \Bigr) \partial_{x_{1}} v^{(0)}(x_1,-l_1,t)
 \\
  +  \Bigl( \eta_1(x_1,t) \big(\Xi^{(1)}_{1}(\xi) - \delta_{1,m} \tfrac{h_0}{h_{1,1}} \xi_2\big)
   + (1-\eta_1(x_1,t) ) \big(\Xi^{(1)}_{2}(\xi) - \delta_{2,m} \tfrac{h_0}{h_{1,2}} \xi_2\big) \Bigr) \partial_{x_2}v^{(0)}(x_1,-l_1,t),
 \\
      \xi_1 = \tfrac{x_1}{\varepsilon}, \  \xi_2 = \tfrac{x_2 + l_1}{\varepsilon},
 \end{multline*}
\begin{multline*}
{\cal N}^{(2)}_m\bigl(\xi, x_1, t \bigr) = \Big(Z^{(2,m)}_{1}(\xi) -  Y_{1,m}(\xi_1) \Bigr) \partial_{x_{1}} v^{(1,m)}(x_1,-l_1-l_2,t)
\\
+  \Bigl( \eta_{2,m}(x_1,t) \Xi^{(2,m)}_{1}(\xi)
  + (1-\eta_{2,m}(x_1,t) ) \Xi^{(2,m)}_{2}(\xi) - \xi_2\Bigr) \partial_{x_2}v^{(1,m)}(x_1,-l_1-l_2,t),
 \\  \xi_1 = \tfrac{x_1}{\varepsilon}, \  \xi_2 = \tfrac{x_2 + l_1 + l_2}{\varepsilon}, \quad m=1, 2;
 \end{multline*}
in {\bf (\ref{4.4})} $Y_{2,m}(\xi_1)= -\xi_1 + b_{2,m} + [\xi_1], \ m=1, 2, 3, 4,$ and
\begin{multline*}
{\cal N}^{(2)}_{2,m}\bigl(\xi, x_1, t \bigr) = \Big(Z^{(2,1)}_{1}(\xi) -  Y_{2,m}(\xi_1)\Bigr) \partial_{x_{1}} v^{(1,1)}(x_1,-l_1-l_2,t)
\\
+
 \Bigl( \eta_{2,1}(x_1,t) \big(\Xi^{(2,1)}_{1}(\xi) - \delta_{1,m}\tfrac{h_{1,1}}{h_{2,1}} \xi_2\big)
 + (1-\eta_{2,1}(x_1,t) ) \big(\Xi^{(2,1)}_{2}(\xi) - \delta_{2,m} \tfrac{h_{1,1}}{h_{2,2}} \xi_2\big) \Bigr) \partial_{x_2}v^{(1,1)}(x_1,-l_1-l_2,t),
 \\
   \xi_1 = \tfrac{x_1}{\varepsilon}, \  \xi_2 = \tfrac{x_2 + l_1 + l_2}{\varepsilon}, \ \ m=1, 2,
 \end{multline*}
\begin{multline*}
{\cal N}^{(2)}_{2,m}\bigl(\xi, x_1, t \bigr) = \Big(Z^{(2,2)}_{1}(\xi) -  Y_{2,m}(\xi_1) \Bigr) \partial_{x_{1}} v^{(1,2)}(x_1,-l_1-l_2,t)
\\
 + \Bigl( \eta_{2,2}(x_1,t) \big(\Xi^{(2,2)}_{1}(\xi) - \delta_{3,m}\tfrac{h_{1,2}}{h_{2,3}} \xi_2\big)
 + (1-\eta_{2,2}(x_1,t) ) \big(\Xi^{(2,2)}_{2}(\xi) - \delta_{4,m} \tfrac{h_{1,2}}{h_{2,4}} \xi_2\big) \Bigr) \partial_{x_2}v^{(1,2)}(x_1,-l_1-l_2,t),
 \\    \xi_1 = \tfrac{x_1}{\varepsilon}, \  \xi_2 = \tfrac{x_2 + l_1 + l_2}{\varepsilon}, \ m=3, 4.
 \end{multline*}

Due to (\ref{trans1'}), (\ref{trans1''}), (\ref{second_trans1}) and (\ref{third_trans1}),
the jumps $\left[ R_\varepsilon \right]|_{Q^{(i)}_\varepsilon}=0,$ $i=0,1, 2.$ This means that
the approximating function $R_\varepsilon$ belongs to $L^{2}\bigl(0,T;\, H^{1}(\Omega _{\varepsilon })\bigr).$

\begin{theorem}\label{th_1}
Suppose that in addition to the assumptions made in Section~\ref{sec1}, the following conditions
hold: the function $f_0 \in  C^1(\overline{\Omega_0}\times [0,T])$ and if some parameter $\beta_i =1$ $(i=0, 1, 2),$ then
the function $g^{(i)}_0 \in  C^1(\overline{D_i}\times [0,T])$ and it and its derivative with respect to $x_2$ vanish at $x_2= - \sum_{n=0}^i l_n$
 and $x_2= - \sum_{n=0}^{i+1} l_n.$

Then for any $\rho\in(0,1)$ there exist positive constants $C_0, \varepsilon_0$ such that for all values $\varepsilon \in
(0,\varepsilon_0)$ the difference between the solution
$v_{\varepsilon}$ to problem (\ref{start_prob})
and the approximating function~$R_{\varepsilon}$ defined by (\ref{4.1}) -- (\ref{4.4})
satisfies the following estimate
\begin{multline}\label{as-est}
 \max_{0 \leq t \leq T}\|R_\varepsilon(\cdot,t) - v_\varepsilon(\cdot,t)\|_{L^2(\Omega_\varepsilon)}
+\|R_\varepsilon - v_\varepsilon\|_{L^2(0,T;H^1(\Omega_\varepsilon))}
    \\
\le C_0 \Big(\varepsilon^{1-\rho}
 +  \sum_{i=0}^2\left(\varepsilon^{\alpha_i -1 + \delta_{\alpha_i, 1}} +
(1- \delta_{\beta_i, 1}) \varepsilon^{\beta_i - 1} + \delta_{\beta_i, 1}\|g^{(i)}_{\varepsilon } - g^{(i)}_{0}\|_{L^2(G^{(i)}_{\varepsilon})}
\right)\Big).
\end{multline}
\end{theorem}

\begin{proof} {\bf I. Residuals in the differential equations.} \
Substituting  $R_{\varepsilon}$ in the differential equations of problem (\ref{start_prob}) instead of $v_\varepsilon$ and calculating discrepancies with regard to problems (\ref{3.3}), (\ref{3.10+1}) -- (\ref{3.10}),  we get
\begin{multline}\label{rez1}
\partial_t R_{\varepsilon} - \Delta_{x}R_{\varepsilon} + k(R_{\varepsilon}) - f_0
=
k(R_{\varepsilon}) - k(v^+) + \varepsilon \chi_0(x_2) \partial_t {\cal N}_+^{(0)}(\xi,x_1,t)|_{\xi=\frac{x}{\varepsilon}}
\\
- \chi^\prime_0(x_2) \bigl( \partial_{\xi_2}{\cal N}_+^{(0)}(\xi,x_1,t) \bigr)\big|_{\xi=\frac{x}{\varepsilon}}
  - \varepsilon \partial_{x_2}\bigl( \chi^\prime_0(x_2) {\cal N}_+^{(0)}(\tfrac{x}{\varepsilon}, x_1, t) \bigr)
\\
-  \chi_0(x_2) \bigl( \partial^2_{x_1\xi_1}{\cal N}_+^{(0)}(\xi,x_1,t)
\bigr)\big|_{\xi=\frac{x}{\varepsilon}}
-   \varepsilon \chi_0(x_2)   \partial_{x_1}\big( \big( \partial_{x_1} {\cal N}_+^{(0)}(\xi,x_1,t)\bigr)\big|_{\xi=\frac{x}{\varepsilon}}
\big)
\\  \text{in} \  \Omega_0\times (0,T);
\end{multline}
\begin{multline}\label{rez2}
\partial_t R_{\varepsilon} - \Delta_{x}R_{\varepsilon} + k_0(R_{\varepsilon})
= k_0(R_{\varepsilon}) - k_0(v^{(0)}) -
2 \delta_{\alpha_0,1} h^{-1}_{0} \kappa_0\big(v^{(0)}\big) + 2 \delta_{\beta_0,1} h^{-1}_{0} g^{(0)}_0
\\
- \varepsilon \partial_{x_1}\Bigl(Y_{0}(\tfrac{x_1}{\varepsilon}) \partial^2_{x_1x_1} v^{(0)} \Bigr)
-  \varepsilon \partial_{x_2}\Bigl(Y_{0}(\tfrac{x_1}{\varepsilon}) \partial^2_{x_2 x_1}v^{(0)}\Bigr)
\\
+ \varepsilon\Bigl( Y_0(\tfrac{x_1}{\varepsilon})\, \partial^2_{t x_1}v^{(0)}(x,t) +
\chi_0(x_2) \partial_{t}{\cal N}^{(0)}_{-}\bigl(\tfrac{x}{\varepsilon}, x_1, t \bigr) + \chi_1(x_2) \partial_{t}{\cal N}^{(1)}\bigl(\tfrac{x_1}{\varepsilon}, \tfrac{x_2 + l_1}{\varepsilon}, x_1, t \bigr) \Bigr)
\\
-   \chi^\prime_0(x_2) \bigl( \partial_{\xi_2}{\cal N}^{(0)}_-(\xi, x_1, t)\bigr)\big|_{\xi=\frac{x}{\varepsilon}}
 - \varepsilon \partial_{x_2}\bigl( \chi^\prime_0(x_2) {\cal N}^{(0)}_-(\tfrac{x}{\varepsilon},x_1, t) \bigr)
\\
-   \chi_0(x_2) \bigl( \partial^2_{x_1\xi_1}{\cal N}^{(0)}_-(\xi, x_1, t)\bigr)\big|_{\xi=\frac{x}{\varepsilon}}
-  \varepsilon \chi_0(x_2)   \partial_{x_1}\big( \big(\partial_{x_1} {\cal N}^{(0)}_-(\xi, x_1, t)\bigr)\big|_{\xi=\frac{x}{\varepsilon}} \big)
\\
-   \chi^\prime_1(x_2) \bigl( \partial_{\xi_2}{\cal N}^{(1)}(\xi, x_1, t)\bigr)\big|_{\xi_1=\frac{x_1}{\varepsilon}, \, \xi_2=\frac{x_2+l_1}{\varepsilon}}
 - \varepsilon \partial_{x_2}\bigl( \chi^\prime_1(x_2) {\cal N}^{(1)}(\tfrac{x_1}{\varepsilon},\tfrac{x_2+l_1}{\varepsilon},x_1, t) \bigr)
\\
-   \chi_1(x_2) \bigl( \partial^2_{x_1\xi_1}{\cal N}^{(1)}(\xi, x_1, t)\bigr)\big|_{\xi_1=\frac{x_1}{\varepsilon}, \, \xi_2=\frac{x_2+l_1}{\varepsilon}}
-   \varepsilon \chi_1(x_2)   \partial_{x_1}\big( \big(\partial_{x_1} {\cal N}^{(1)}(\xi, x_1, t)\bigr)\big|_{\xi_1=\frac{x_1}{\varepsilon}, \, \xi_2=\frac{x_2+l_1}{\varepsilon}} \big)
\\
\text{in} \ \  G^{(0)}_\varepsilon\times (0,T);
\end{multline}
and similar relations in $G^{(1,m)}_\varepsilon\times (0,T), \ (m=1, 2)$ and  $G^{(2,m)}_\varepsilon\times (0,T), \ (m=1, 2, 3, 4)$ up to replacement of indices.

\medskip

{\bf II. Residuals in the boundary and initial conditions.} \
Obviously, $R_\varepsilon |_{t=0} =0.$
Also using (\ref{trans1'}), (\ref{trans1''}), (\ref{second_trans2}), (\ref{3.23}), (\ref{third_trans2}), (\ref{third_trans4}), (\ref{3.24}) and  (\ref{3.25}), one can verify that
\begin{gather}
\left[\partial_{x_2}R_\varepsilon \right]\big|_{Q^{(0)}_\varepsilon} = - \varepsilon Y_0(\tfrac{x_1}{\varepsilon})\, \partial^2_{x_1 x_2}v^{(0)}(x_1,0,t) ,
\notag
\\
\left[\partial_{x_2}R_\varepsilon \right]\big|_{Q^{(1,m)}_\varepsilon} =  \varepsilon
\Big(Y_0  \partial^2_{x_1 x_2}v^{(0)}(x,t) -  Y_{1,m} \partial^2_{x_1 x_2}v^{(1,m)}(x,t)\Big)|_{x_2=-l_1}, \  m=1, 2, \notag
\\
\left[\partial_{x_2}R_\varepsilon \right]\big|_{Q^{(2,m)}_\varepsilon} =  \varepsilon
\Big(Y_{1,1} \partial^2_{x_1 x_2}v^{(1,1)}(x,t)  -  Y_{2,m}
\partial^2_{x_1 x_2}v^{(2,m)}(x,t)\Big)\big|_{x_2=-l_1 -l_2}, \ m=1, 2, \notag
\\
\left[\partial_{x_2}R_\varepsilon \right]\big|_{Q^{(2,m)}_\varepsilon} =  \varepsilon
\Big(Y_{1,2} \partial^2_{x_1 x_2}v^{(1,1)}(x,t)  -  Y_{2,m} \partial^2_{x_1 x_2}v^{(2,m)}(x,t)\Big)|_{x_2=-l_1 -l_2}, \  m=3, 4, \label{jump_deriv}
\end{gather}
where $Q^{(i,m)}_{\varepsilon} = \partial G^{(i,m)}_{\varepsilon} \cap \{x_2= - \sum_{n=1}^i l_n\}.$

Since $Z^{(0)}_1$ is odd in $\xi_1$ and  $Z^{(0)}_2$ is even in $\xi_1$ (see Proposition~\ref{prop1}), it is easy to check that $\partial_{\nu}R_\varepsilon=0$ on $\partial\Omega_\varepsilon \cap \{x: x_2 \ge 0\}.$ In additional,  one can verify that
\begin{gather}
\partial_{x_2}R_\varepsilon\big|_{\partial\Omega_\varepsilon \cap \{x_2 = -l_1\}} =  \varepsilon
\, Y_0(\tfrac{x_1}{\varepsilon})\, \partial^2_{x_1 x_2}v^{(0)}(x_1, -l_1, t), \notag
\\
\partial_{x_2}R_\varepsilon\big|_{\partial\Omega_\varepsilon \cap \{x_2 = -l_1 -l_2\} \cap \partial G^{(1,m)}_\varepsilon}
=  \varepsilon\, Y_{1,m}(\tfrac{x_1}{\varepsilon})\, \partial^2_{x_1 x_2}v^{(1,m)}(x_1,-l_1- l_2,t), \ m=1, 2, \notag
\\
\partial_{x_2}R_\varepsilon|_{\partial \Omega_\varepsilon \cap \{x_2 = -l_1-l_2 -l_3\}} = 0. \label{B_1}
\end{gather}

Taking into account boundary conditions in problems (\ref{sol_Z_i}), (\ref{3.13}), (\ref{3.15'}), we find the values
of $\partial_{x_1}R_\varepsilon$ on the vertical boundary of the branches:
 \begin{multline}\label{B_2}
\partial_{x_1}R_\varepsilon = \varepsilon \Big(Y_{0}(\tfrac{x_1}{\varepsilon})\, \partial^2_{x_1 x_1}v^{(0)}(x,t)
+
\chi_0(x_2)\big(\partial_{x_1}{\cal N}^{(0)}_{-}(\xi, x_1, t) \big)\big|_{\xi=\tfrac{x}{\varepsilon}}
\\
+ \chi_1(x_2)
\big(\partial_{x_1}{\cal N}^{(1)}(\xi, x_1, t) \big)\big|_{\xi_1=\tfrac{x_1}{\varepsilon},\,  \xi_2 =\tfrac{x_2 + l_1}{\varepsilon}} \Bigr)
\quad \text{on} \ \partial G^{(0)}_\varepsilon \cap \{x_2 \in(-l_1,0)\},
\end{multline}
\begin{multline}\label{B_3}
\partial_{x_1}R_\varepsilon = \varepsilon \Big(Y_{1,m}(\tfrac{x_1}{\varepsilon})\, \partial^2_{x_1 x_1}v^{(1,m)}(x,t)
+
\chi_1(x_2)\big(\partial_{x_1}{\cal N}^{(1)}_{1,m}(\xi, x_1, t) \big)\big|_{\xi_1=\tfrac{x_1}{\varepsilon}, \, \xi_2 =\tfrac{x_2 + l_1}{\varepsilon}}
\\
+ \chi_2(x_2)
\big(\partial_{x_1}{\cal N}^{(2)}_m(\xi, x_1, t) \big)\big|_{x_1=\tfrac{x_1}{\varepsilon}, \,x_2 =\tfrac{x_2 + l_1 + l_2}{\varepsilon}} \Bigr)
\\  \text{on} \ \partial G^{(1,m)}_\varepsilon \cap \{x_2 \in(-l_1-l_2,-l_1)\}, \ \ m=1, 2,
\end{multline}
\begin{multline}\label{B_4}
\partial_{x_1}R_\varepsilon = \varepsilon \Big(Y_{2,m}(\tfrac{x_1}{\varepsilon})\, \partial^2_{x_1 x_1}v^{(2,m)}(x,t)
+
\chi_2(x_2)
\big(\partial_{x_1}{\cal N}^{(2)}_{2,m}(\xi, x_1, t) \big)\big|_{x_1=\tfrac{x_1}{\varepsilon}, \, x_2 =\tfrac{x_2 + l_1 + l_2}{\varepsilon}} \Bigr)
\\
\text{on} \ \partial G^{(2,m)}_\varepsilon \cap \{x_2 \in(-l_1-l_2-l_3,-l_1-l_2)\}, \ \
\quad m=1, 2, 3, 4.
\end{multline}

\medskip

{\bf III. Residuals in the integral identity.} \
Multiplying (\ref{rez1}) and (\ref{rez2}) for each indexes $i$ and $m$  with arbitrary function
$\psi \in L^2(0, T; H^{1}(\Omega _{\varepsilon })),$  integrating by parts and taking
(\ref{jump_deriv})--(\ref{B_4}) into account, we deduce
\begin{equation}\label{res_int-1}
\int\limits_{\Omega _{\varepsilon }}\partial _{t} R_{\varepsilon}\,\psi \, dx +
\langle {\cal A}_\varepsilon(t) R_{\varepsilon} , \psi \rangle_\varepsilon
=  \int\limits_{\Omega _{0}}f_{0}\,\psi \, dx
+ {\cal F}_\varepsilon(\psi)
\end{equation}
for a.e. $t\in (0,T].$ Subtracting the integral identity (\ref{1.4}) from (\ref{res_int-1}) and integrating over $t \in (0, \tau),$ where $\tau \in (0,T],$
we get
\begin{multline}\label{res_int-2}
\int_{0}^{\tau}\Big( \langle R'_{\varepsilon} - v'_\varepsilon, \psi   \rangle_\varepsilon
+ \langle {\cal A}_\varepsilon(t) R_{\varepsilon} - {\cal A}_\varepsilon(t) v_{\varepsilon} , \psi \rangle_\varepsilon
\Big) \, dt
\\
 = \int_{0}^{\tau} \Big({\cal F}_\varepsilon(\psi) - \sum_{i=0}^2 \varepsilon ^{\beta_i}
\int\limits_{\Upsilon_{\varepsilon }^{(i)}} g^{(i)}_{\varepsilon }\,\psi \, dx_2\Big) \, dt,
\end{multline}
where ${\cal F}_\varepsilon(\psi) = \sum_{j=1}^5 {\cal I}^\varepsilon_j(\psi)$ and (to short formulas we omit variables $\frac{x}{\varepsilon}, x, t$ in some places)
$$
{\cal I}^\varepsilon_1(\psi)= \int\limits_{\Omega _{0}} \big(k(R_\varepsilon) - k(v^+)\big) \, \psi \, dx + \sum_{i=0}^2 \sum_{m=1}^{2i} \int\limits_{G^{(i,m)}_{\varepsilon}} \big(k_i(R_\varepsilon) - k_i(v^{(i,m)})\big) \,\psi \, dx,
$$
$$
{\cal I}^\varepsilon_2(\psi) = \sum_{i=0}^2 \sum_{m=1}^{2i}\Bigg( \varepsilon^{\alpha_i} \int\limits_{\Upsilon_{\varepsilon }^{(i,m)}} \kappa_i(R_\varepsilon) \, \psi \, dx_2
 - 2 \delta_{\alpha_i,1} h^{-1}_{i,m}  \int\limits_{G^{(i,m)}_{\varepsilon}}  \kappa_i\big(v^{(i,m)}\big)
 \, \psi\, dx \Bigg),
$$
$$
{\cal I}^\varepsilon_3(\psi) = 2 \sum_{i=0}^2 \sum_{m=1}^{2i} \delta_{\beta_i,1} h^{-1}_{i,m} \int\limits_{G^{(i,m)}_{\varepsilon}} g^{(i)}_0 \, \psi\, dx,
$$
$$
{\cal I}^\varepsilon_4(\psi) = \varepsilon \Bigg( \int\limits_{\Omega_0}\Big(\chi_0(x_2) \partial_t {\cal N}_+^{(0)} \, \psi
+ \chi^\prime_0(x_2) {\cal N}_+^{(0)}\, \partial_{x_2} \psi
 + \chi_0(x_2)  \big( \partial_{x_1} {\cal N}_+^{(0)}(\xi,x_1,t)\bigr)\big|_{\xi=\frac{x}{\varepsilon}} \, \partial_{x_1} \psi \Big) \, dx
$$
$$
+ \sum_{i=0}^2 \sum_{m=1}^{2i} \int\limits_{G^{(i,m)}_{\varepsilon}} Y_{i,m}(\tfrac{x_1}{\varepsilon})\big( \partial^2_{x_2 x_1}v^{(i,m)}\, \partial_{x_2}\psi +  \partial^2_{x_1 x_1}v^{(i,m)}\, \partial_{x_1}\psi \big) \, dx
$$
$$
+ \int\limits_{G^{(0)}_{\varepsilon}}
\Bigl( Y_0(\tfrac{x_1}{\varepsilon})\, \partial^2_{t x_1}v^{(0)} +
\chi_0(x_2) \partial_{t}{\cal N}^{(0)}_{-}\bigl(\tfrac{x}{\varepsilon}, x_1, t \bigr) + \chi_1(x_2) \partial_{t}{\cal N}^{(1)}\bigl(\tfrac{x_1}{\varepsilon}, \tfrac{x_2 + l_1}{\varepsilon}, x_1, t \bigr) \Bigr) \, \psi \, dx
$$
$$
+ \int\limits_{G^{(0)}_{\varepsilon}} \Big(\chi_0(x_2)   \big(\partial_{x_1} {\cal N}^{(0)}_-(\xi, x_1, t)\bigr)\big|_{\xi=\frac{x}{\varepsilon}} + \chi_1(x_2) \big(\partial_{x_1} {\cal N}^{(1)}(\xi, x_1, t)\bigr)\big|_{\xi_1=\frac{x_1}{\varepsilon}, \, \xi_2=\frac{x_2+l_1}{\varepsilon}}
 \Big)\partial_{x_1}\psi \, dx
$$
$$
+ \int\limits_{G^{(0)}_{\varepsilon}} \Big( \chi^\prime_0(x_2) {\cal N}^{(0)}_-(\tfrac{x}{\varepsilon},x_1, t) +
\chi^\prime_1(x_2) {\cal N}^{(1)}(\tfrac{x_1}{\varepsilon},\tfrac{x_2+l_1}{\varepsilon},x_1, t) \Big)\partial_{x_2}\psi \, dx
$$
$$
+ \sum_{m=1}^2\int\limits_{G^{(1,m)}_{\varepsilon}}
\Bigl( Y_{1,m}(\tfrac{x_1}{\varepsilon})\, \partial^2_{t x_1}v^{(1,m)} +
\chi_1\, \partial_{t}{\cal N}^{(1)}_{1,m}\bigl(\tfrac{x_1}{\varepsilon}, \tfrac{x_2 + l_1}{\varepsilon}, x_1, t \bigr) + \chi_2\,
\partial_{t}{\cal N}^{(2)}_m\bigl(\tfrac{x_1}{\varepsilon}, \tfrac{x_2 + l_1+ l_2}{\varepsilon}, x_1, t \bigr) \Bigr) \, \psi \, dx
$$
$$
+ \int\limits_{G^{(1,m)}_{\varepsilon}} \Big(\chi_1 \, \big(\partial_{x_1} {\cal N}^{(1)}_{1,m}(\xi, x_1, t)\bigr)\big|_{\xi_1=\frac{x_1}{\varepsilon}, \, \xi_2=\frac{x_2+l_1}{\varepsilon}} + \chi_2\, \big(\partial_{x_1} {\cal N}^{(2)}_m(\xi, x_1, t)\bigr)\big|_{\xi_1=\frac{x_1}{\varepsilon}, \, \xi_2=\frac{x_2+l_1 + l_2}{\varepsilon}}
 \Big)\partial_{x_1}\psi \, dx
$$
$$
+ \int\limits_{G^{(1,m)}_{\varepsilon}} \Big( \chi^\prime_1(x_2) {\cal N}^{(1)}_{1,m}(\tfrac{x_1}{\varepsilon},\tfrac{x_2+l_1}{\varepsilon},x_1, t) +
\chi^\prime_2(x_2) {\cal N}^{(2)}_m(\tfrac{x_1}{\varepsilon},\tfrac{x_2+l_1+l_2}{\varepsilon},x_1, t) \Big)\partial_{x_2}\psi \, dx
$$
$$
+ \sum_{m=1}^4\int\limits_{G^{(2,m)}_{\varepsilon}}
\Bigl( Y_{2,m}(\tfrac{x_1}{\varepsilon})\, \partial^2_{t x_1}v^{(2,m)} + \chi_2(x_2) \,
\partial_{t}{\cal N}^{(2)}_{2,m}\bigl(\tfrac{x_1}{\varepsilon}, \tfrac{x_2 + l_1+ l_2}{\varepsilon}, x_1, t \bigr) \Bigr) \, \psi \, dx
$$
$$
+ \int\limits_{G^{(2,m)}_{\varepsilon}}  \chi_2(x_2) \, \big(\partial_{x_1} {\cal N}^{(2)}_{2,m}(\xi, x_1, t)\bigr)\big|_{\xi_1=\frac{x_1}{\varepsilon}, \, \xi_2=\frac{x_2+l_1 + l_2}{\varepsilon}} \, \partial_{x_1}\psi \, dx
$$
$$
+ \int\limits_{G^{(2,m)}_{\varepsilon}} \chi^\prime_2(x_2) {\cal N}^{(2)}_{2,m}(\tfrac{x_1}{\varepsilon},\tfrac{x_2+l_1+l_2}{\varepsilon},x_1, t) \, \partial_{x_2}\psi \, dx \Bigg),
$$

$$
{\cal I}^\varepsilon_5(\psi) = - \int\limits_{\Omega_0} \chi^\prime_0(x_2) \bigl( \partial_{\xi_2}{\cal N}_+^{(0)}(\xi,x_1,t) \bigr)\big|_{\xi=\frac{x}{\varepsilon}} \, \psi \, dx
$$
$$
- \int\limits_{G^{(0)}_{\varepsilon}} \Big(\chi^\prime_0(x_2) \bigl( \partial_{\xi_2}{\cal N}^{(0)}_-(\xi, x_1, t)\bigr)\big|_{\xi=\frac{x}{\varepsilon}} +  \chi^\prime_1(x_2) \bigl( \partial_{\xi_2}{\cal N}^{(1)}(\xi, x_1, t)\bigr)\big|_{\xi_1=\frac{x_1}{\varepsilon}, \, \xi_2=\frac{x_2+l_1}{\varepsilon}} \Big) \psi \, dx
$$
$$
- \sum_{m=1}^2\int\limits_{G^{(1,m)}_{\varepsilon}} \Big(\chi^\prime_1 \bigl( \partial_{\xi_2}{\cal N}^{(1)}_{1,m}(\xi, x_1, t)\bigr)\big|_{\xi_1=\frac{x_1}{\varepsilon}, \, \xi_2=\frac{x_2+l_1}{\varepsilon}} +  \chi^\prime_2 \bigl( \partial_{\xi_2}{\cal N}^{(2)}_m(\xi, x_1, t)\bigr)\big|_{\xi_1=\frac{x_1}{\varepsilon}, \, \xi_2=\frac{x_2+l_1+l_2}{\varepsilon}} \Big) \psi \, dx
$$
$$
- \sum_{m=1}^4\int\limits_{G^{(2,m)}_{\varepsilon}} \chi^\prime_2(x_2)\,  \bigl( \partial_{\xi_2}{\cal N}^{(2)}_{2,m}(\xi, x_1, t)\bigr)\big|_{\xi_1=\frac{x_1}{\varepsilon}, \, \xi_2=\frac{x_2+l_1+l_2}{\varepsilon}} \, \psi \, dx,
$$

$$
{\cal I}^\varepsilon_6(\psi) = - \int\limits_{\Omega_0} \chi_0(x_2) \bigl( \partial^2_{x_1\xi_1}{\cal N}_+^{(0)}(\xi,x_1,t)
\bigr)\big|_{\xi=\frac{x}{\varepsilon}} \psi \, dx
$$
$$
- \int\limits_{G^{(0)}_{\varepsilon}} \Big(\chi_0(x_2) \bigl( \partial^2_{x_1 \xi_1}{\cal N}^{(0)}_-(\xi, x_1, t)\bigr)\big|_{\xi=\frac{x}{\varepsilon}} +  \chi_1(x_2) \bigl( \partial^2_{x_1 \xi_1}{\cal N}^{(1)}(\xi, x_1, t)\bigr)\big|_{\xi_1=\frac{x_1}{\varepsilon}, \, \xi_2=\frac{x_2+l_1}{\varepsilon}} \Big) \psi \, dx
$$
$$
- \sum_{m=1}^2\int\limits_{G^{(1,m)}_{\varepsilon}} \Big(\chi_1 \bigl( \partial^2_{x_1 \xi_1}{\cal N}^{(1)}_{1,m}(\xi, x_1, t)\bigr)\big|_{\xi_1=\frac{x_1}{\varepsilon}, \, \xi_2=\frac{x_2+l_1}{\varepsilon}} +  \chi_2 \bigl( \partial^2_{x_1 \xi_1}{\cal N}^{(2)}_m(\xi, x_1, t)\bigr)\big|_{\xi_1=\frac{x_1}{\varepsilon}, \, \xi_2=\frac{x_2+l_1+l_2}{\varepsilon}} \Big) \psi \, dx
$$
$$
- \sum_{m=1}^4\int\limits_{G^{(2,m)}_{\varepsilon}} \chi_2(x_2)\,  \bigl( \partial^2_{x_1 \xi_1}{\cal N}^{(2)}_{2,m}(\xi, x_1, t)\bigr)\big|_{\xi_1=\frac{x_1}{\varepsilon}, \, \xi_2=\frac{x_2+l_1+l_2}{\varepsilon}} \, \psi \, dx.
$$

Let us estimate the right-hand side in (\ref{res_int-2}).  Due to the conditions  (\ref{cond-2}) we have $|{\cal I}^\varepsilon_1(\psi)| \le C_1 \varepsilon \|\psi\|_{L^2(\Omega_\varepsilon)}.$ To estimate $|{\cal I}^\varepsilon_2(\psi)|,$ we use special integral identities
\begin{equation}\label{identity1}
  \frac{\varepsilon h_{i,m}}{2}\int\limits_{\Upsilon_{\varepsilon }^{(i,m)} } \phi
  \: dx_2=\int\limits_{G_\varepsilon ^{(i,m)}} \phi \: dx -
  \varepsilon\int\limits_{G^{(i,m)}_{\varepsilon}}Y_{i,m} \left(\frac{x_1}{\varepsilon
}\right)\partial_{x_1} \phi \: dx \quad \forall\: \phi\in
H^1\bigl(G^{(i,m)}_\varepsilon \bigr),
\end{equation}
for $i\in \{0,1,2\}, \ m =\overline{1, 2i}.$ To prove (\ref{identity1}) it is enough to integrate by parts the last integral in (\ref{identity1}).
If $\alpha_i=1,$ then with the help of (\ref{identity1}) we deduce
\begin{multline}\label{e-1}
\Bigg| \varepsilon^{1} \int\limits_{\Upsilon_{\varepsilon }^{(i,m)}} \kappa_i(R_\varepsilon) \, \psi \, dx_2
 - 2  h^{-1}_{i,m}  \int\limits_{G^{(i,m)}_{\varepsilon}}  \kappa_i\big(v^{(i,m)}\big)
 \, \psi\, dx \Bigg|
 \\
 \le 2  h^{-1}_{i,m}  \int\limits_{G^{(i,m)}_{\varepsilon}}  \Big| \kappa_i(R_\varepsilon) - \kappa_i\big(v^{(i,m)}\big)\Big|
  \big| \psi \big| \, dx
   \, + \,  \varepsilon \int\limits_{G^{(i,m)}_{\varepsilon}}\Big| Y_{i,m} \left(\frac{x_1}{\varepsilon}\right)\Big|
   \big|\partial_{x_1}\big(\kappa_i(R_\varepsilon)  \psi \big) \big| \, dx
\\
 \le C_2 \varepsilon \|\psi\|_{H^1(\Omega_\varepsilon)}.
\end{multline}
In the last inequality we use (\ref{cond-2}), (\ref{est-2}) and inequality $\max_{\Bbb R}|Y_{i,m}| \le 1.$
If $\alpha_i >1,$ then again with the help of (\ref{identity1}) we get
\begin{equation*}
\Bigg| \varepsilon^{\alpha_i} \int_{\Upsilon_{\varepsilon }^{(i,m)}} \kappa_i(R_\varepsilon) \, \psi \, dx_2\Bigg|
 \le C_3 \varepsilon^{\alpha_i -1} \|\psi\|_{H^1(\Omega_\varepsilon)}.
\end{equation*}
Therefore, $|{\cal I}^\varepsilon_2(\psi)| \le C_4 \sum_{i=0}^2\varepsilon^{\alpha_i -1 + \delta_{\alpha_i, 1}} \|\psi\|_{H^1(\Omega_\varepsilon)}.$

Similar, but now using (\ref{1.2}) and (\ref{1.3}), we can estimate
\begin{multline*}
\Bigg| {\cal I}^\varepsilon_3(\psi) -  \sum_{i=0}^2 \varepsilon ^{\beta_i}
\int\limits_{\Upsilon_{\varepsilon }^{(i)}} g^{(i)}_{\varepsilon }\,\psi \, dx_2\Bigg|
\\
\le C_5 \|\psi\|_{H^1(\Omega_\varepsilon)}
\sum_{i=0}^2 \Big( (1- \delta_{\beta_i, 1}) \varepsilon^{\beta_i - 1} + \delta_{\beta_i, 1}\big(\|g^{(i)}_{\varepsilon } - g^{(i)}_{0}\|_{L^2(G^{(i,m)}_{\varepsilon}) } + \varepsilon\big)\Big).
\end{multline*}

It is easy to see that ${\cal I}^\varepsilon_4(\psi)$ is of order ${\cal O}(\varepsilon).$ Thanks to the asymptotic estimates
(\ref{w2}),  (\ref{w3}),  (\ref{3.14}) -- (\ref{3.16}), all integrals in ${\cal I}^\varepsilon_5(\psi)$  are integrated
 over the support of the functions $\{\chi'_i\}_{i=0}^2.$ Therefore, they are exponentially small.

Since the functions $\partial^2_{x_1\xi_1}{\cal N}_+^{(0)},$  $\partial^2_{x_1\xi_1}{\cal N}_-^{(0)},$ $\partial^2_{x_1\xi_1}{\cal N}^{(1)},$
$\{\partial^2_{x_1\xi_1}{\cal N}_{1,m}^{(1)}, \partial^2_{x_1\xi_1}{\cal N}^{(2)}_m\}_{m=1}^{2},$
$\{\partial^2_{x_1\xi_1}{\cal N}_{2,m}^{(2)}\}_{m=1}^{4}$ exponentially decrease as $|\xi_2| \to + \infty$ (see
(\ref{w2}),  (\ref{w3}),  (\ref{3.14}) -- (\ref{3.16})), we deduce from Lemma 3.1 (\cite{MN96}) that for any $\rho \in (0,1)$
the integrals in ${\cal I}^\varepsilon_6(\psi)$ are of order ${\cal O}(\varepsilon^{1-\rho}).$

Regarding to the inequalities obtained above in this subsection, we conclude that for the right-hand side in (\ref{res_int-2})
for every $\tau \in (0,T]$ the following inequality holds
\begin{multline}\label{est-2}
  \Bigg| \int\limits_{0}^{\tau} \Big({\cal F}_\varepsilon(\psi) - \sum_{i=0}^2 \varepsilon ^{\beta_i}
\int\limits_{\Upsilon_{\varepsilon }^{(i)}} g^{(i)}_{\varepsilon }\,\psi \, dx_2\Big) \, dt \Bigg|
\le
C \|\psi\|_{L^2(0,T; H^1(\Omega_\varepsilon))} \,
\\
\times \Big(\varepsilon^{1-\rho}  +  \sum_{i=0}^2\big(\varepsilon^{\alpha_i -1 + \delta_{\alpha_i, 1}} +
(1- \delta_{\beta_i, 1}) \varepsilon^{\beta_i - 1} + \delta_{\beta_i, 1}\|g^{(i)}_{\varepsilon } - g^{(i)}_{0}\|_{L^2(G^{(i)}_{\varepsilon})}
\big)\Big).
\end{multline}

Putting $R_\varepsilon - v_\varepsilon$ instead $\psi$  in (\ref{res_int-1}) and taking into account that ${\cal A}_\varepsilon$ is strictly monotone,
we derive from  (\ref{res_int-1}) and (\ref{est-2}) the estimate (\ref{as-est}).
\end{proof}

\begin{remark}
The constant $C_0$  in  (\ref{as-est}) depends on the following quantities:
$$
\sup_{(x_1,t)\in (0, a)\times(0,T)}\,  \bigl|\partial^2_{t x_j}v^+(x_1,0, t)\bigr|,
\quad
\sup_{(x_1,t)\in (0,a)\times(0,T)}\,  \bigl| \mathcal{D}^{\alpha}v^+(x_1,0,t) \bigr|,
$$
$$
\sup_{(x_1,t)\in (0, a)\times(0,T)}\,  \bigl|\partial^2_{t x_j}v^{(0)}(x_1,-l_1, t)\bigr|,
\quad
\sup_{(x_1,t)\in (0,a)\times(0,T)}\,  \bigl| \mathcal{D}^{\alpha}v^{(0)}(x_1,-l_1,t)\bigr|,
$$
$$
\sup_{(x_1,t)\in (0, a)\times(0,T)}\,  \bigl|\partial^2_{t x_j}v^{(1,m)}(x_1,- l_1 - l_2, t)\bigr|,
\
\sup_{(x_1,t)\in (0,a)\times(0,T)}\,  \bigl| \mathcal{D}^{\alpha} v^{(1,m)}(x_1,- l_1 - l_2, t) \bigr|,
$$
$m=1, 2, \ j=1, 2,$
and $\|\partial^2_{t x_1}v^{(i,m)}_0\|_{L^2(D_i\times(0,T))},$ where $i\in \{0, 1, 2\},$ $m=\overline{1, 2i},$
$ \  |\alpha|=\alpha_1+\alpha_2 \le 2.$
\ Due to the assumptions for the functions $f_0$ and $\{g^{(i)}_0\}_{i=0}^2$ and
condition (\ref{cond-2}) it follows from classical results on the smoothness of solutions to
semilinear parabolic problems (see for instance \S 6 and \S 7 from \cite[Sec. V]{Lad}) that these quantities  are bounded.
\end{remark}

From Theorem~\ref{th_1} it follows directly the Corollary~\ref{corollary}.

\end{document}